\documentclass[11pt]{article}
\usepackage{pifont}
\usepackage{amsmath, amsthm, amscd, amsfonts, graphicx,makeidx,amssymb,mathrsfs, float,capt-of}
\usepackage[all]{xy}
\usepackage{verbatim}
\usepackage[mathcal]{euscript}
\usepackage{latexsym}
\usepackage{epsfig}

\usepackage{rotating}
\usepackage{pdfsync}

\newtheorem{thm}{Theorem}[section]
\newtheorem{cor}[thm]{Corollary}
\newtheorem{lem}[thm]{Lemma}
\newtheorem{prop}[thm]{Proposition}
\newtheorem{rem}[thm]{Remark}
\theoremstyle{definition}
\newtheorem{defn}[thm]{\textbf{Definition}}

\theoremstyle{definition}

\newtheorem{example}{Example}[section]
\theoremstyle{remark}

\newcommand{\ben}{\begin{equation}}
\newcommand{\een}{\end{equation}}

\newcommand{\VV}{{\mathcal V}}
\newcommand{\WW}{{\mathcal W}}

\newcommand{\ga}{\gamma}

\newcommand{\ve}{\varepsilon}
\newcommand{\de}{\delta}

\newcommand{\rt}{\rightarrow}
\newcommand{\m}{\mathcal{W}}

\newcommand{\ma}{\mathcal}
\newcommand{\mb}{\mathbb}

\newdir{ >}{{}*!/-10pt/@{>}}

\newcommand{\ov}{\overline}

\newcommand{\ot}{\otimes}

\title{The Lie bialgebra structure of the vector space of cyclic words.}
\author{Ana Gonz\'alez\thanks{Centro de Matem\'atica de la Facultad de Ciencias  (Montevideo, Uruguay).{\em e-mail: }{\tt ana@cmat.edu.uy }} \\
}
\begin{document}

\date{November 20, 2011}

\maketitle

\begin{abstract}
In this paper we give a combinatory proof of the Lie bialgebra structure presented in the vector space of reduced cyclic words. This structure was introduce by M. Chas in \cite{moira}, where the proof of the existence of this Lie bialgebra structure is based on the existence of an isomorphism between the space of reduced cyclic word and the space of curves on a surface.
\end{abstract}

\tableofcontents

\section{Introduction}
\hspace{.5cm}The purpose of this work is to give a combinatory proof of the existence of a Lie bialgebra structure in the space of reduced cyclic words $\mb{V}$. In \cite{moira} M. Chas gave a proof of this result by means of an isomorphism of Lie bialgebras between $\mb{V}$ and the space of curves on a surface. In \cite{gold} Goldman proved that the space of curves on a surface admits a Lie algebra structure. Subsequently, Turaev in \cite{turaev} defined the Lie coalgebra structure on $\mb{V}$ and he proved the compatibility between the algebra and the coalgebra structures.

In this paper I will prove this theorem (namely, the Lie algebra structure of $\mb{V}$) using purely combinatorial methods.

The organization of this paper is as follows. In the second section I present the concept of Lie bialgebra and introduce some basic examples. The third section is dedicated to the construction of the space of cyclic words $\mb{V}$, and to setting up the terminology and main definitions that allow us to define the Lie algebra and coalgebra associated to this space. In the fourth section I prove that $\mb{V}$ admits a Lie algebra structure. The final two sections are dedicated to give the Lie coalgebra structure of $\mb{V}$ and to prove the compatibility between these two structures.

\section{Lie bialgebras}
\hspace{.5cm} The infinitesimal notion of a Hopf algebra is a Lie bialgebra consisting of a pair $(\ma{A},\de)$, where $\ma{A}$ is a Lie algebra and $\de:\ma{A}\longrightarrow \ma{A}\otimes\ma{A}$ is a linear map called cobracket. It satisfies some axioms dual to those of a Lie algebra together with a compatibility condition with the Lie algebra structure.\\

Let $A$ denote a $\Bbbk$-module. In order to recall the definition of a Lie bialgebra we need two auxiliar linear maps. These are
$$s:A\otimes A\rightarrow A\otimes A \quad \mathrm{and}\quad \ve:A\otimes A\otimes A\rightarrow A\otimes A \otimes A$$ defined by
$$s(x\otimes y)=y\otimes x\quad \mathrm{and}\quad\ve(x\otimes y\otimes z)= z\otimes x\otimes y.$$

\begin{defn}
A \emph{Lie algebra} is given by a $\Bbbk$-module $A$ and a linear map $[\; ,\; ]:A\otimes A\rightarrow A$ such that
$[\;,\;]\circ s=-[\;,\;]$ (skew symmetry) and $[\;,\;]\circ\bigl(id\otimes [\;,\;]\bigr)\circ\bigl(id+\ve+\ve^2\bigr)=0$ (Jacobi identity).
\end{defn}
\begin{defn}
A \emph{Lie coalgebra} is given by a $\Bbbk$-module $A$ and a linear map $\de:A\rt A\otimes A$ such that  $s\circ \de =-\de$ (coskew symmetry) and $\bigl(id+\ve+\ve^2\bigr)\circ\bigl(id\otimes\de\bigr)\circ\de=0$ (co-Jacobi identity).
\end{defn}
\begin{defn}
The triple $\bigl(A,[\;,\;],\de\bigr)$  is a \emph{Lie bialgebra} if $\bigl(A,[\;,\;]\bigr)$ is a Lie algebra, $(A,\de)$ is a Lie coalgebra and the compatibility equation $\de\bigl([x,y]\bigr)=x\cdot\de(y)-y\cdot\de(x)$ holds for every $x,y\in A$, where
$x\cdot(y\otimes z)=[x,y]\otimes z+ y\otimes[x,z]$.
\end{defn}
\begin{defn}
The triple $\bigl(A,[\;,\;],\de\bigr)$ is an \emph{involutive} Lie bialgebra if  $\bigl(A,[\;,\;],\de\bigr)$ is a Lie bialgebra and $[\;,\;]\circ\de=0$ on $A$.

\end{defn}

\subsection{Examples}
\hspace{.5cm}Now we will present some examples of Lie bialgebras. We refer the reader to \cite{majid} for details.

\subsubsection{The 2-dimensional real Lie bialgebra.}
\hspace{.7cm}Let $B=\bigl\langle H,X\bigr\rangle$ the vector space over the complex numbers generated by  $H$ and $X$.\\
We define the linear and skew symmetry map  $[\, ,\,]:B\otimes B\rightarrow B$ by $\bigl[H,X\bigr]=2X$ and the linear
and coskew symmetry map $\de:B\rightarrow B\otimes B$ by $\de(H)=0$ and $\de(X)=\frac{1}{2}\bigl(X\otimes H-H\otimes X\bigr)$.
Then $\bigl(B,[\, ,\,],\de\bigr)$ is a Lie bialgebra.

This Lie bialgebra is self-dual and it is a sub-Lie bialgebra (in the obvious sense) of the following importan example.

\subsubsection{The 3-dimensional complex Lie bialgebra $sl_2$.}
\hspace{.7cm}Let us consider  $sl_2=\bigl\langle H,X_\pm\bigr\rangle$ the $\mb{C}$-vector
space generated by $H$, $X_+$ and $X_-$ with brackets $\bigl[X_+,X_-\bigr]=H$,
$\bigl[H,X_\pm\bigr]=\pm 2X_\pm$ and cobrackets  $\de(H)=0$,
$\de\bigl(X_\pm\bigr)=\frac{1}{2}\bigl(X_\pm\otimes H-H\otimes X_\pm\bigr)$. Then
$\bigl(sl_2,[\, ,\,],\de\bigr)$ is a Lie  bialgebra.

This Lie bialgebra is not self-dual. Let $\bigl\{\Phi,\Psi_\pm\bigr\}$ be the dual basis of $\bigl\{H,X_\pm\bigr\}$. Nevertheless we can find the dual Lie bialgebra as follows.

\subsubsection{The dual Lie bialgebra of $sl_2$.}
\hspace{.7cm}Let $sl_2^*=\bigl\langle \Phi,\Psi_\pm\bigr\rangle$  the dual space of $sl_2$
generated by $\Phi$, $\Psi_+$ and $\Psi_-$. The bracket $[\,
,\,]:sl_2^*\otimes sl_2^*\rightarrow sl_2^*$ is defined by
$\bigl[\Psi_\pm,\Phi\bigr]=\frac {1}{2}\Psi_\pm$, $\bigl[\Psi_+,\Psi_-\bigr]=0$ and
the cobracket $\de:sl_2^*\rightarrow sl_2^*\otimes sl_2^*$ is
defined by $\de\bigl(\Psi\pm\bigr)=\pm 2\bigl(\Phi \otimes\Psi_\pm
-\Psi_\pm\otimes\Phi\bigr)$, $\de(\Phi)=\Psi_+\otimes \Psi_-
-\Psi_-\otimes \Psi_+$. Then $\bigl(sl_2^*,[\, ,\,],\de\bigr)$ is a Lie
bialgebra.

\section{The vector space of cyclic words $\mathbb{V}$.}
\hspace{0.5cm}In this section we construct, from a finite set of symbols, the vector space $\mb{V}$, called the vector space of cyclic word.  In the following sections we will prove that this space admits an additional structure, this is a Lie bialgrebra structure.

\begin{figure}[htp]
    \centering
    \includegraphics{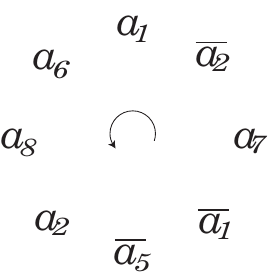}
    \caption{A cyclic word in the letters of $\mathbb{A}_8$.}\label{fig1}
   \end{figure}

\begin{defn}
For each non-negative integer $n$, the $n$-\emph{alphabet} is the
set of $2n$ symbols $\mathbb{A}_n =\bigl\{a_1, a_2,\dots
,a_n,\overline{a_1},\overline{a_2},\dots ,\overline{a_n}\bigr\}.$ We
will consider linear words, denoted with capital roman characters,
and cyclic words, denoted by capital calligraphical characters, both
constructed with letters from $\mathbb{A}_n$.

 A \emph{cyclic word} is a linear word that additionally satisfies the cyclic condition.
 We can think the cyclic words as symbols placed at the vertices of
the $n$-th root of unit in $\mathbb{C}$ up to circular symmetry,
$n=1,2,3,\dots$ see Figure \ref{fig1}.

\end{defn}

 Let $x_0x_1\dots x_m$ be a linear word, by definition, $\overline{x_0x_1\dots x_m} = \overline{x_m}\dots\overline{x_1}\ \overline{x_0}$ and for each letter $x$, $\overline{\overline{x}}=x$. A linear word $x_0x_1\dots x_m$ is \emph{freely reduced} if $x_i\neq \overline{x_{i+1}}$ for each $i\in\{0, 1,\dots ,m\}$.

A linear word $W$ is a \emph{linear representative} of a cyclic word $\WW$ if $W$ can be obtained from $\WW$ by making a cut between two consecutive letters of $\WW$. In such a case, we write $\WW= c(W)$. If
$\WW$ is a cyclic word, $W$ is a linear representative of $\WW$, and $n$ is a positive integer, we define $\WW^n$ as $c\bigl(W^n\bigr)$, $\overline{\WW}$ as $c\bigl(\overline{W}\bigr)$, and $\WW^{-n}$ as $\overline{\WW}^n$.

\begin{defn}
A cyclic word is \emph{reduced} if it is non-empty and all its linear representatives are freely reduced.\\
A reduced cyclic word is \emph{primitive} if it cannot be written as $\WW^r$ for some $r\geq 2$ and some reduced cyclic word $\WW$. The \emph{length} of a linear (respectively cyclic) word $W$ (respectively $\WW$) is the number of letters counted with multiplicity that it contains and it is denoted by $l(W)$ (resp. $l(\WW)$).\\
By a \emph{subword} of a cyclic word $\WW$, we mean a linear subword of one of the linear representatives of $\WW$.
\end{defn}

Let $\mathcal{O}$ be a reduced cyclic word such that every letter of $\mathbb{A}_n$ appears exactly once. The word $\mathcal{O}$ is called a \emph{surface symbol}.

\begin{defn}
To each cyclic word $\mathcal{W}$, we associate a number, $o(\mathcal{W})\in\{-1,0,1\}$ as follows.
\begin{enumerate}
\item[(i)] If $\mathcal{W}$ is reduced and there exists an injective orientation preserving  map, from the letters of $\mathcal{W}$ to the letters of $\mathcal{O}$ then $o(\mathcal{W})=1$.
\item[(ii)] If $\mathcal{W}$ is reduced and there exists an injective orientation reversing  map, from the letters of $\mathcal{W}$ to the letters of $\mathcal{O}$ then $o(\mathcal{W})=-1$.
\item[(iii)]  In all other cases (that is, if $\mathcal{W}$ is not reduced or if there is no such orientation preserving or reversing map) $o(\mathcal{W}) = 0$.
\end{enumerate}
\end{defn}

\begin{defn}\label{d1}
Let $P$ and $Q$ be two linear words. The ordered pair $(P,Q)$ is $\mathcal{O}$-\emph{linked} if $P$ and $Q$ are reduced words of length at least two and one of the following conditions holds:
\begin{enumerate}
\item[(1)] $P=p_1p_2$, $Q=q_1q_2$ and   $o\bigl(c(\ov{p}_1\ov{q}_1p_2q_2)\bigr)\neq 0$;
\item[(2)] $P=p_1Yp_2$, $Q=q_1Yq_2$, $p_1\neq q_1$, $p_2\neq q_2$ and $Y$ is a linear word of length at least one and if $Y=x_1Xx_2$, then $o\bigl(c(\ov{p}_1\ov{q}_1x_1)\bigr)=o\bigl(c(p_2q_2\ov{x}_2)\bigr)$;
\item[(3)] $P=p_1Yp_2$, $Q=q_1\ov{Y}q_2$, $p_1\neq  \ov{q}_2$, $p_2\neq \ov{q}_1$ and $Y$ is a linear word of length at least one and if $Y=x_1Xx_2$, then $o\bigl(c(q_2\ov{p}_1x_1)\bigr)=o\bigl(c(\ov{q}_1p_2\ov{x}_2)\bigr)$.
\end{enumerate}
\end{defn}

Let $\WW$ be a reduced cyclic word and denote by $\mathbb{LP}_1(\WW)$ the set of linked pairs $(P,Q)$, for $P$ and $Q$ are linear subwords of $\WW$.

\begin{defn}
We denote by $\mathbb{V}$ the vector space generated by non-empty reduced cyclic words with letters in $\mathbb{A}_n$.
\end{defn}

A possible geometric interpretation of the definition of linked pair by means of curves on a surface is as
follows.

For talking of linked pair it is necessary to fix a surface symbol, that is, a cyclic word $\mathcal{O}=c\bigl(o_1o_2\dots o_{2n}\bigr)$ such that
every letter of $\mathbb{A}_n$ appears exactly once. Denote by $P_{\mathcal{O}}$ the $4n$-gon with edges labeled counterclockwise
in the following way: one chooses an edge in first place and labels it with $o_1$, the second lacks of label, the third is labeled with
$o_2$, the fourth lacks of label and so on as is shown in the example of figure \ref{fig3}.
\begin{figure}[htp]
    \centering
    \includegraphics{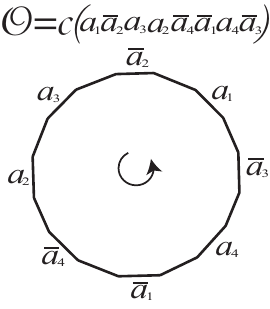}
    \caption{The $4n$-gon $P_{\mathcal{O}}$.}\label{fig3}
   \end{figure}\\
For each $i\in\bigl\{1,2,\dots ,n\bigr\}$, one identifies the edge $a_i$ with the edge $\ov{a_i}$ without creating Moebius bands. In this way, one
gets a surface $\Sigma_{\mathcal{O}}$ with non-empty boundary and Euler characteristic $(1-n)$. Furthermore, every surface with
non-empty boundary can be obtained from such a $4n$-gon. Denote by $\pi:P_{\mathcal{O}}\longrightarrow \Sigma_{\mathcal{O}}$ the obvious
projection map. A \emph{loop} in $\Sigma_{\mathcal{O}}$ is a piecewise smooth map from the circle to $\Sigma_{\mathcal{O}}$. In
these conditions, the set $\bigl\{a_1,a_2,\dots ,a_n\bigr\}$ is a generator of the fundamental group of $\Sigma_{\mathcal{O}}$. If $\WW$ is a
reduced cyclic word in the letters of $\mathbb{A}_n$, a loop $\alpha$ in $\Sigma_{\mathcal{O}}$ is a \emph{representative of}
$\WW$ if $\alpha$ is freely homotopic to the curve $\beta$ where the homotopy class of $\beta$ written in the generators $\bigl\{a_1,a_2,\dots ,a_n\bigl\}$ is a linear representative of $\WW$.\\
The idea of linked pairs is the following: two threads on a surface come close, stay together for some time and then separate. If one thread enters the strip from above and exits below and the other vice versa we must have an intersection. This is measured by linked pairs (see Figure \ref{fig4}).
\begin{figure}[htp]
    \centering
    \includegraphics{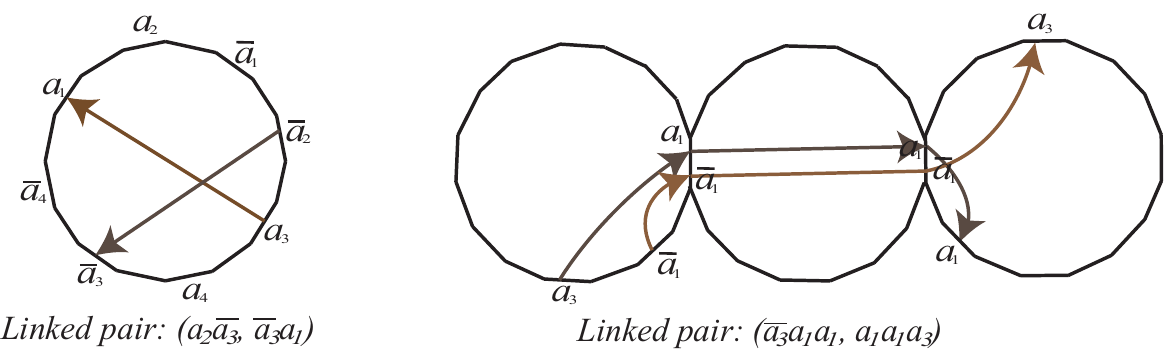}
    \caption{Geometric interpretation of linked pairs of type $(1)$ and type $(2)$.}\label{fig4}
   \end{figure}

\section{The Lie coalgebra structure}
\hspace{.5cm}In this section we define the coalgebra structure on the vector space of reduced cyclic words $\mb{V}$.

\begin{defn}\label{def1}
For each ordered pair $(P,Q)\in\mathbb{LP}_1(\m)$ we associate two cyclic words $\delta_1(P,Q)=c(W_1)$ and $\delta_2(P,Q)=c(W_2)$ as follows:
\begin{enumerate}
\item[(i)] Assume that (1) or (2) of definition \ref{d1} holds. Make two cuts on $\WW$, one immediately before $p_2$ and the other immediately before $q_2$. We obtain two linear words, $W_1$ and $W_2$, the first,
starting at $p_2$, and the second, starting at $q_2$.
\item[(ii)] If condition (3) holds, let $W_1$ be the linear subword of $\WW$ starting at $p_2$ and ending at $q_1$, and let $W_2$ be the linear subword of $\WW$ starting at $q_2$ and ending at $p_1$.
\end{enumerate}
\end{defn}
\begin{figure}[htp]
    \centering
    \includegraphics{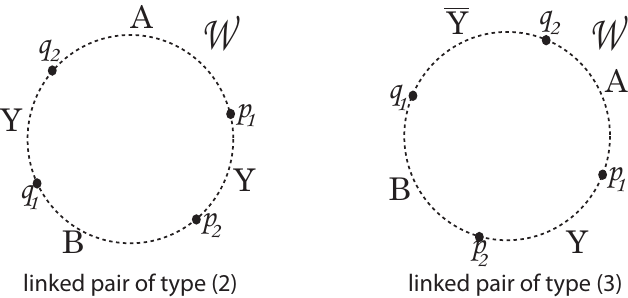}
    \caption{Cyclic words associated to the linked pair $(P,Q)\in\mathbb{LP}_1(\WW)$.}\label{fig2}
   \end{figure}

\begin{lem}\label{le1}
Let $\WW$ be a cyclic reduced word. For each $(P,Q)\in\mathbb{LP}_1(\m)$, the linear words $W_1$ and $W_2$ of the above definition are disjoint in $\WW$. Moreover, $W_1$ and $W_2$ are non-empty and one
can write $\WW= c\bigl(W_1W_2\bigr)$ in the case (i) above and $\WW= c\bigl(YW_1\overline{Y}W_2\bigr)$ in the case (ii).
\end{lem}
\begin{prop}
Let $\WW$ be a reduced cyclic word and let $(P,Q)\in\mathbb{LP}_1(\WW)$ be a linked pair. Then $\de_1(P,Q)$ and $\de_2(P,Q)$ (from definition \ref{def1}) are reduced cyclic words. Moreover, $\de_1(P,Q)$ and $\de_2(P,Q)$ are non-empty.
\end{prop}

\begin{defn}
 To each linked pair $(P,Q)$ one associates a \emph{sign} as follows:
\begin{enumerate}
\item[(i)] If $(P,Q)$ is a linked pair of type (1), then $sign(P,Q):=o\bigl(c(\overline{p_1}\ \overline{q_1}p_2q_2)\bigr)$;
\item[(ii)] if  $(P,Q)$ is a linked pair of type (2), then $sign(P,Q):=o\bigl(c(\overline{p_1}\ \overline{q_1}x_1)\bigr)$;
\item[(iii)] if  $(P,Q)$ is a linked pair of type (3), then $sign(P,Q):=o\bigl(c(q_2\overline{p_1}x_1)\bigr).$
\end{enumerate}
\end{defn}
\begin{lem}\label{le2}
\begin{enumerate}
\item[(a)] For every linked pair $(P,Q)$, $sign(P,Q) = 1$ or $sign(P,Q) = -1$.
\item[(b)] If $(P,Q)$ is a linked pair, then $(Q,P)$ is also a linked pair. Moreover, $sign(P,Q)=-sign(Q,P)$.
\end{enumerate}
\end{lem}

The previous results were proved in \cite{moira}.

\begin{defn}
We define $\de:\mathbb{V}\longrightarrow\mathbb{V}\otimes\mathbb{V}$ as the linear map such that for every reduced cyclic word $\WW$,
$$\de(\WW)=\sum_{(P,Q)\in\mathbb{LP}_1(\WW)}sign(P,Q)\de_1(P,Q)\otimes\de_2(P,Q).$$
This sum is finite as a consequence that the set $\mathbb{LP}_1(\WW)$ is finite.
\end{defn}
\begin{thm}\label{t1}
$(\mathbb{V},\de)$ is a Lie coalgebra.
\end{thm}
To give a combinatory proof of this result we need a series of technical lemmas.

\begin{lem}\label{le2.4.1}
Let $\ma{W}$ be a reduced cyclic word. If
$(P,Q)\in\mathbb{LP}_1(\m)$, then $\de_1(P,Q)=\de_2(Q,P)$ and
$\de_2(P,Q)=\de_1(Q,P)$.
\end{lem}
\begin{proof}
We only need to prove the first identification because the second follows immediately by considering the pair $(Q,P)$ instead of $(P,Q)$.

For $\de_1(\WW)=c(W_1)$ and $\de_2(P,Q)=c(W_2)$ as in definition \ref{def1}.
\begin{enumerate}
\item[(i)] If $(P,Q)$ is a linked pair of type $(1)$, then we have $W_1=p_2Bq_1$ and  $W_2=q_2Ap_1$. Consequently $\de_1(P,Q)=c(p_2Bq_1)$. By the other hand, applying the Lemma \ref{le2}, $(Q,P)\in\mathbb{LP}_1(\WW)$ and $\de_2(Q,P)=c(V_2)$ where $V_2=p_2Bq_1$. Finally, $\de_2(Q,P)=c\bigl(p_2Bq_1\bigr)=\de_1(P,Q)$.
\item[(ii)] If $(P,Q)$ is a linked pair of type $(2)$ then $W_1=p_2Bq_1Y$ and $W_2=q_2Ap_1Y$. Therefore $\de_1(P,Q)=c(p_2Bq_1Y)$. As $(Q,P)\in\mathbb{LP}_1(\WW)$ we have $\de_2(Q,P)=c(V_2)$ where $V_2=p_2Bq_1Y$. Then $\de_2(Q,P)=c\bigl(p_2Bq_1Y\bigr)=\de_1(P,Q)$.
\item[(iii)]  Finally, if $(P,Q)$ is a linked pair of type $(3)$, we have $W_1=p_2Bq_1$, $W_2=q_2Ap_1$ and $\de_1(P,Q)=c\bigl(p_2Bq_1\bigr)$. As before $(Q,P)\in\mathbb{LP}_1(\WW)$ and $\de_2(Q,P)=c(V_2)$ where $V_2=p_2Bq_1$. Then $\de_2(Q,P)=c\bigl(p_2Bq_1\bigr)=\de_1(P,Q)$.
\end{enumerate}
\end{proof}

\begin{prop}\label{p2}
For $\ma{W}$ a reduced cyclic word the cobracket $\delta$ satisfies the following identity
 $$\de(\WW)=\sum_{\bigl\{(P,Q),(Q,P)\bigr\}\subset\mathbb{LP}_1(\m)}sign(P,Q)\bigl\{\de_1(P,Q)\otimes\de_2(P,Q)-\de_2(P,Q)\otimes\de_1(P,Q)\bigr\}.$$
\end{prop}
\begin{proof}
$$\begin{array}{rcl}
   \de(\WW)&= & \displaystyle{\sum_{(P,Q)\in\mathbb{LP}_1(\WW)}sign(P,Q)\de_1(P,Q)\otimes\de_2(P,Q)} \\
    & = &\displaystyle{ \sum_{\bigl\{(P,Q),(Q,P)\bigr\}\subset\mathbb{LP}_1(\WW)}sign(P,Q)\de_1(P,Q)\otimes\de_2(P,Q)+sign(Q,P)\de_1(Q,P)\otimes\de_2(Q,P)} \\
    & = & \displaystyle{\sum_{\bigl\{(P,Q),(Q,P)\bigr\}\subset\mathbb{LP}_1(\WW)}sign(P,Q)\bigl\{\de_1(P,Q)\otimes\de_2(P,Q) -\de_1(Q,P)\otimes\de_1(Q,P)\bigr\} }\\
    & = & \displaystyle{\sum_{\bigl\{(P,Q),(Q,P)\bigr\}\subset\mathbb{LP}_1(\WW)}sign(P,Q)\bigl\{\de_1(P,Q)\otimes\de_2(P,Q) -\de_2(P,Q)\otimes\de_1(P,Q)\bigr\}}
 \end{array}$$
 \end{proof}

For $(P,Q)\in\mathbb{LP}_1(\WW)$ we take $(R,S)\in\mathbb{LP}_1\bigl(\de_i(P,Q)\bigr)$, where $i\in\{1,2\}$. We can cut $\de_i(P,Q)$ by $(R,S)$ and construct the cyclic words $\de_j(R,S)_{\de_i(P,Q)}$, where $j\in\{1,2\}$. Denote by $\de_{ji}(R,S)$ these new cyclic words.

\begin{prop}\label{p3}
Let $\WW$ be a cyclic reduced word. For each $(P,Q)\in\mathbb{LP}_1(\WW)$ and $(R,S)\in\mathbb{LP}_1\bigl(\de_i(P,Q)\bigr)$ the linked pair $(P,Q)$ belongs to the set $\mathbb{LP}_1\bigl(\de_2(R,S)\bigr)$.
\end{prop}
\begin{proof}
First, we note that if $(P,Q)$ is a linked pair of type $(1)$, then in both cases $(P,Q)\in\mathbb{LP}_1\bigl(\de_2(R,S)\bigr)$. Now, we study the other two cases.

We present the proof when  $(R,S)$ is a linked pair of type $(2)$, The other case is left to the reader.

 If $(P,Q)$ is a linked pair of type $(2)$, where $P=p_1Xp_2$ and $Q=q_1Xq_2$, then $\de_1(P,Q)=c(p_2Bq_1X)$ and $\de_2(P,Q)=c(q_2Ap_1X)$.
Let $R=r_1Yr_2$ and $S=s_1Ys_2$.
If $(R,S)\in\mathbb{LP}_1(\de_1(P,Q))$, with $\de_1(P,Q)=c\bigl(p_2B_1r_1Yr_2Cs_1Ys_2B_2q_1X\bigr)$, we have that $\de_{1}(R,S)=c\bigl(r_2Cs_1Y\bigr)$ and $\de_{2}(R,S)=c\bigl(s_2B_2q_1Xq_2Ap_1Xp_2B_1r_1Y\bigr)$ then $(P,Q)\in\mathbb{LP}_1(\de_2(R,S))$.
 On the other hand, if we have that $(R,S)\in\mathbb{LP}_1\bigl(\de_2(P,Q)\bigr)$, with $\de_2(P,Q)=c\bigl(q_2A_1r_1Yr_2A_2s_1Ys_2A_3p_1X\bigr)$, then $\de_{1}(R,S)=c\bigl(r_2A_2s_1Y\bigr)$ and $\de_{2}(R,S)=c\bigl(s_2A_3p_1Xp_2Bq_1Xq_2A_1r_1Y\bigr)$, therefore $(P,Q)\in\mathbb{LP}_1\bigl(\de_2(R,S)\bigr)$.

If $(P,Q)$ is a linked pair of type $(3)$, where $P=p_1Xp_2$ and $Q=q_1\ov{X}q_2$, then $\de_1(P,Q)=c\bigl(p_2Bq_1\bigr)$ and $\de_2(P,Q)=c\bigl(q_2Ap_1\bigr)$. Let $R=r_1Yr_2$ and $S=s_1Ys_2$.
If we have that $(R,S)\in\mathbb{LP}_1\bigl(\de_1(P,Q)\bigr)$ and $\de_1(P,Q)=c\bigl(p_2B_1r_1Yr_2Cs_1Ys_2B_2q_1\bigr)$, then $\de_{1}(R,S)=c\bigl(r_2Cs_1Y\bigr)$ and $\de_{2}(R,S)=c\bigl(s_2B_2q_1\ov{X}q_2Ap_1Xp_2B_1r_1Y\bigr)$, therefore $(P,Q)\in\mathbb{LP}_1\bigl(\de_2(R,S)\bigr)$. On the other hand, if $(R,S)\in\mathbb{LP}_1\bigl(\de_2(P,Q)\bigr)$ and $\de_2(P,Q)=c\bigl(q_2A_1r_1Yr_2A_2s_1Ys_2A_3p_1\bigr)$, therefore $\de_{1}(R,S)=c\bigl(r_2A_2s_1Y\bigr)$ and $\de_2(R,S)=c\bigl(s_2A_3p_1Xp_2Bq_1\ov{X}q_2A_1r_1Y\bigr)$, then $(P,Q)\in\mathbb{LP}_1\bigl(\de_2(R,S)\bigr)$.
\end{proof}

\begin{cor}\label{c1}
Under the same hypothesis as in the previous proposition the following holds
\begin{enumerate}
\item If $(P,Q)\in\mathbb{LP}_1\bigl(\de_2(R,S)\bigr)$ and $(R,S)\in\mathbb{LP}_1\bigl(\de_1(P,Q)\bigr)$, then $\left\{\begin{array}{ccc}
                                                                                     \de_1(R,S) & = & \de_{11}(R,S) \\
                                                                                     \de_2(P,Q) & = & \de_{22}(P,Q) \\
                                                                                     \de_{12}(P,Q) & = & \de_{21}(R,S)
                                                                                   \end{array}
                                                                                           \right.
$
\item If $(P,Q)\in\mathbb{LP}_1\bigl(\de_2(R,S)\bigr)$ and $(R,S)\in\mathbb{LP}_1\bigl(\de_2(P,Q)\bigr)$, then $\left\{\begin{array}{ccc}
                                                                                     \de_1(R,S) & = & \de_{12}(R,S) \\
                                                                                     \de_1(P,Q) & = & \de_{12}(P,Q) \\
                                                                                     \de_{22}(P,Q) & = & \de_{22}(R,S)
                                                                                   \end{array}\right.
$

\end{enumerate}
\end{cor} The proof of this corollary is an immediate consequence of the last proposition.

\begin{proof} (Theorem \ref{t1})
\begin{enumerate}
\item[(1)] Coskew symmetry: $s\circ \de =-\de$,
\begin{alignat*}{2}
s\circ \de(\m)
&=s\left(\sum_{(P,Q)\in\mathbb{LP}_1(\m)}sign(P,Q)\delta_1(P,Q)\otimes\delta_2(P,Q)\right)&\\
&=\sum_{(P,Q)\in\mathbb{LP}_1(\m)}sign(P,Q)\delta_2(P,Q)\otimes\delta_1(P,Q)&\\
&=\sum_{(Q,P)\in\mathbb{LP}_1(\m)}-sign(Q,P)\delta_2(P,Q)\otimes\delta_1(P,Q)&\\
&=-\sum_{(Q,P)\in\mathbb{LP}_1(\m)}sign(Q,P)\delta_1(Q,P)\otimes\delta_2(Q,P)&\\
&=-\de(\m).&
\end{alignat*}

\item[(2)] Co-Jacobi identity: $\bigl(id+\ve+\ve^2\bigr)\circ\bigl(id\otimes\de\bigr)\circ\de=0$, applying the Proposition \ref{p2}
\begin{alignat*}{2}
\de(\m) & =\sum_{\bigl\{(P,Q),(Q,P)\bigr\}\subset\mathbb{LP}_1(\m)}sign(P,Q)\bigl\{\de_1(P,Q)\otimes\de_2(P,Q)-\de_2(P,Q)\otimes\de_1(P,Q)\bigr\}.\\
\bigl(id\otimes\de\bigr)\bigl(\de(\m)\bigr)&=\sum_{\bigl\{(P,Q),(Q,P)\bigr\}\subset\mb{LP}_1(\m)}\sum_{\bigl\{(R,S),(S,R)\bigr\}\subset\mb{LP}_1(\de_2(P,Q))} sign(P,Q)sign(R,S)\de_1(P,Q)\otimes\\
&\de_{12}(R,S)\otimes\de_{22}(R,S)-sign(P,Q)sign(R,S)\de_1(P,Q)\otimes\de_{22}(R,S)\otimes\de_{12}(R,S)\\
&-\sum_{\bigl\{(P,Q),(Q,P)\bigr\}\subset\mb{LP}_1(\m)}\sum_{\bigl\{(R,S),(S,R)\bigr\}\subset\mb{LP}_1(\de_1(P,Q))}sign(P,Q)sign(R,S)\de_2(P,Q)\otimes\\
&\de_{11}(R,S)\otimes\de_{21}(R,S)-sign(P,Q)sign(R,S)\de_2(P,Q)\otimes\de_{21}(R,S)\otimes\de_{11}(R,S).
\end{alignat*}
For $(P,Q)\in\mathbb{LP}_1(\m)$ and $(R,S)\in\mb{LP}_1\bigl(\de_2(P,Q)\bigr)$,
by Proposition \ref{p3} we have that $(P,Q)\in
\mb{LP}_1\bigl(\de_2(R,S)\bigr)$. The term of $\bigl(id\otimes\de\bigr)\bigl(\de(\m)\bigr)$, which
is associated to the pair $\bigl\{(P,Q),(R,S)\bigr\}$, is
\begin{alignat*}{2}
\bigl(id\otimes\de_{(R,S)}\bigr)\bigl(\de_{(P,Q)}(\m)\bigr)&= sign(P,Q)sign(R,S)\de_1(P,Q)\otimes\de_{12}(R,S)\otimes\de_{22}(R,S)\\
&-sign(P,Q)sign(R,S)\de_1(P,Q)\otimes\de_{22}(R,S)\otimes\de_{12}(R,S).
\end{alignat*}
By the other hand, the term of $\bigl(id\otimes\de\bigr)\bigl(\de(\m)\bigr)$, which is
associated to the pair $\{(R,S),(P,Q)\}$, is
\begin{alignat*}{2}
\bigl(id\otimes\de_{(P,Q)}\bigr)\bigl(\de_{(R,S)}(\m)\bigr)&=sign(R,S)sign(P,Q)\de_1(R,S)\otimes\de_{12}(P,Q)\otimes\de_{22}(P,Q)\\
&-sign(R,S)sign(P,Q)\de_1(R,S)\otimes\de_{22}(P,Q)\otimes\de_{12}(P,Q).
\end{alignat*}
If we apply the Corollary \ref{c1}, then the last equality can be
replaced by
\begin{alignat*}{2}
\bigl(id\otimes\de_{(P,Q)}\bigr)\bigl(\de_{(R,S)}(\m)\bigr)&=sign(R,S)sign(P,Q)\de_{12}(R,S)\otimes\de_{1}(P,Q)\otimes\de_{22}(R,S)\\
&-sign(R,S)sign(P,Q)\de_{12}(R,S)\otimes\de_{22}(R,S)\otimes\de_{1}(P,Q).
\end{alignat*}
Finally, applying $id+\varepsilon+\varepsilon^2$ we have that
$$\bigl(id+\varepsilon+\varepsilon^2\bigr)\bigl(id\otimes\de_{(R,S)}\bigr)\bigl(\de_{(P,Q)}(\m)\bigr)+\bigl(id+\varepsilon+\varepsilon^2\bigr)\bigl(id\otimes\de_{(P,Q)}\bigr)\bigl(\de_{(R,S)}(\m)\bigr)=0.$$

If $(P,Q)\in\mathbb{LP}_1(\m)$ and $(R,S)\in\mb{LP}_1\bigl(\de_1(P,Q)\bigr)$, then
\begin{alignat*}{2}
\bigl(id\otimes\de_{(R,S)}\bigr)\bigl(\de_{(P,Q)}(\m)\bigr)=& -sign(P,Q)sign(R,S)\de_2(P,Q)\otimes\de_{11}(R,S)\otimes\de_{21}(R,S)\\
&+sign(P,Q)sign(R,S)\de_2(P,Q)\otimes\de_{21}(R,S)\otimes\de_{11}(R,S).
\end{alignat*}
By the other hand, the term of $\bigl(id\otimes\de\bigr)\bigl(\de(\m)\bigr)$, which is associated to the pair $\bigl\{(R,S),(P,Q)\bigr\}$, is
\begin{alignat*}{2}
\bigl(id\otimes\de_{(P,Q)}\bigr)\bigl(\de_{(R,S)}(\m)\bigr)&=sign(R,S)sign(P,Q)\de_1(R,S)\otimes\de_{12}(P,Q)\otimes\de_{22}(P,Q)\\
&-sign(R,S)sign(P,Q)\de_1(R,S)\otimes\de_{22}(P,Q)\otimes\de_{12}(P,Q).
\end{alignat*}
If we apply the Corollary \ref{c1}, then the last equality can be
replaced by
\begin{alignat*}{2}
\bigl(id\otimes\de_{(P,Q)}\bigr)\bigl(\de_{(R,S)}(\m)\bigr)&=sign(R,S)sign(P,Q)\de_{11}(R,S)\otimes\de_{21}(R,S)\otimes\de_{2}(P,Q)\\
&-sign(R,S)sign(P,Q)\de_{11}(R,S)\otimes\de_{2}(P,Q)\otimes\de_{21}(R,S).
\end{alignat*}
Finally, applying $id+\varepsilon+\varepsilon^2$ we have that
$$\bigl(id+\varepsilon+\varepsilon^2\bigr)\bigl(id\otimes\de_{(R,S)}\bigr)\bigl(\de_{(P,Q)}(\m)\bigr)+(id+\varepsilon+\varepsilon^2)\bigl(id\otimes\de_{(P,Q)}\bigr)\bigl(\de_{(R,S)}(\m)\bigr)=0.$$

\end{enumerate}
\end{proof}

\section{The Lie algebra structure}
\hspace{.5cm}Let $\VV$ and $\WW$ be two cyclic words. The set of linked pairs of $\VV$ and $\WW$, denoted by $\mathbb{LP}_2(\VV,\WW)$, is defined as the set of all pairs $(P,Q)$ for whose there exist positive integers $j$
and $k$ such that $P$ is an occurrence of a subword of $\VV^j$ and $Q$ is an occurrence of a subword of $\WW^k$, where $l\bigl(\VV^{j-1}\bigr)<l(P)\leq l\bigl(\VV^j\bigr)$ and $l\bigl(\WW^{k-1}\bigr)<l(Q)\leq l\bigl(\WW^k\bigr)$. \\ 

The reader can see the proof of the next results in \cite{moira}.
\begin{prop}\label{p1}
Let $\VV$ and $\WW$ be reduced cyclic words. Then we have the following consequences:
\begin{enumerate}
\item[(a)] there are at most $l(\VV)l(\WW)$ elements in $\mathbb{LP}_2(\VV,\WW)$,
\item[(b)] the set of linked pairs of one word, $\mathbb{LP}_1(\WW)$, contains at most $l(\WW)\bigl(l(\WW)-1\bigr)$ elements.
\end{enumerate}
\end{prop}

\begin{lem}
If $P=x_0x_1\dots x_{m-1}$ is a linear word and for some $i\in\bigl\{1,2,\dots, m-1\bigr\}$, $P=x_ix_{i+1}\dots x_{m-1}x_0x_1\dots x_{i-1}$ then there exists a linear word $Q$ and an integer $r$ such that $r\geq 2$, $P=Q^r$ and $l(Q)$
divides $i$.
\end{lem}

\begin{lem}\label{l3}
Let $\VV$ and $\WW$ be cyclic words which are not powers of the same cyclic word and let $P$ be a linear word. Let $k$ and $l$ be positive integers such that $P$ is a subword of $\VV^k$ and either $P$ or $\overline{P}$ is a subword of $\WW^l$. Moreover, assume that $(k-1)l(\VV)<l(P)$ and $(l-1)l(\WW)<l(P)$. Then $l(P)<l(\VV) + l(\WW)$.
\end{lem}

\begin{rem}\label{r2}
Let be $\ma{V}$ and $\ma{W}$ reduced cyclic words and
$(P,Q)\in\mb{LP}_2(\ma{V},\ma{W})$. Then we have the following consequences:
\begin{enumerate}
\item[i)] If $l(\ma{V})=l(\ma{W})$, then $P\subset \ma{V}$ and $Q\subset\ma{W}$.
\item[ii)] If $l(\ma{V})< l(\ma{W})$, then $W_1$ is not a power of $V_1$.
\end{enumerate}
\end{rem}
\begin{proof}
\begin{enumerate}
\item[i)] Note that, in these conditions, if $P\subset\ma{V}^k$, then  $Q\subset\ma{W}^k$.\\
Suppose that $k\neq 1$ as a consequence $P=p_1Xp_2$ with $X=BV_1^{k-1}$ and
$Q=q_1Xq_2$, where $X=CW_1^{k-1}$. As $l(V_1)=l(W_1)$ and $X=BV_1^{k-1}=CW_1^{k-1}$ is $V_1=W_1$,
in particular we have  that $p_2=q_2$ and this is a contradiction. Consequently, $k=1$.
\item[ii)] Let $\ma{V}$ and $\ma{W}$ be cyclic words such that $l(\ma{V})<l(\ma{W})$. If $W_1=V_1^k$, then $W_1=q_2A=V_1^k$ where  $V_1=p_2B$.
Finally, since $W_1=p_2BV_1^{k-1}$ is $q_2=p_2$ which contradicts
that $(P,Q)$ is a linked pair.
\end{enumerate}
\end{proof}

\begin{prop}\label{p2}
Let $\ma{V}$ and $\ma{W}$ be reduced cyclic words. Then
$\mb{LP}_2(\ma{V},\ma{W})$ can be defined as the set of all linked pairs $(P,Q)$
such that $P$ is an occurrence of a subword of $\ma{V}^j$, $Q$ is an
occurrence of a subword of $\ma{W}^k$,  $l\bigl(\ma{V}^{j-1}\bigr)< l(P)\leq
l\bigl(\ma{V}^j\bigr)$, $l\bigl(\ma{W}^{k-1}\bigr)<l(Q)\leq l\bigl(\ma{W}^k\bigr)$ where $j$ and
$k$ are positive integers such that $\displaystyle{j\leq 2+\frac{l(\ma{W})}{l(\ma{V})}}$ and $\displaystyle{k\leq 2+\frac{l(\ma{V})}{l(\ma{W})}}$.
\end{prop}
\begin{proof}
 For $(P,Q)\in \mb{LP}_2(\ma{V},\ma{W})$ we have that $P$ is an occurrence of a subword of $\ma{V}^j$ and $Q$ is an occurrence of a subword of $\ma{W}^k$, for some positive integers $j$ and $k$. We can choose $j$ and $k$ to be the minimal positive integers such that $l\bigl(\ma{V}^{j-1}\bigr)<l(P)\leq l\bigl(\ma{V}^j\bigr)$ and $l\bigl(\ma{W}^{k-1}\bigr)<l(Q)\leq l\bigl(\ma{W}^k\bigr)$.\\
Suppose that $(P,Q)$ is a linked pair of type $(2)$  (the case (3) is analogous). \\
$(1)$ For $l(\ma{V})=l(\ma{W})$, by remark \ref{r2}, we have that $P\subset\ma{V}$ and $Q\subset\ma{W}$. For this reason, we have that $j=k=1$.\\
$(2)$ Suppose that $l(\ma{V})<l(\ma{W})$. By the lemma \ref{l3},  $l(Q)=2+l(X)<2+l(\ma{V})+l(\ma{W})$. Since $l(\ma{V})<l(\ma{W})$ we have that $l(Q)\leq 2 l(\ma{W})$. Consequently $\displaystyle{k\leq 2+\frac{l(\ma{V})}{l(\ma{W})}}$.\\
Similarly for $Q$, we have that $l(P)=2+l(X)\leq l(\ma{V})+ l(\ma{W})+1$. Moreover, $l\bigl(\ma{V}^{j-1}\bigr)<l(P)$, and  we conclude that $(j-1)l(\ma{V})<l(P)\leq l(\ma{V})+l(\ma{W})+1$. That  is
$(j-1)l(\ma{V}) \leq l(\ma{V})+l(\ma{W})$. As a consequence $\displaystyle{j\leq 2+\frac{l(\ma{W})}{l(\ma{V})}}$.
\end{proof}

In the next example we see that the given upper bound is optimal.
\begin{example}
Let be $\mathcal{O}=c\bigl(a_1a_2\ov{a_1}\,\ov{a_2}\bigr)$, $\mathcal{V}=c\bigl(a_1a_1a_2\bigr)$ and $\WW=c\bigl(a_1a_1a_2a_1a_1a_2a_1\bigr)$. We construct the linked pair $(P,Q)\in\mathbb{LP}_2(\mathcal{V},\WW)$
as $$\begin{array}{ccc}
  P & = & a_2a_1a_1a_2a_1a_1a_2a_1a_1a_2, \\
  Q & = & a_1a_1a_1a_2a_1a_1a_2a_1a_1a_1.
\end{array}
$$
In this case we have that $V_1=a_2a_1a_1$ and $W_1=a_2a_1a_1a_2a_1a_1a_1$, we have the following
$$
\begin{array}{rcl}
  c\bigl(V_1^4\bigr) & = & c\bigl(a_1a_1a_2a_1a_1a_2a_1a_1a_2a_1a_1a_2\bigr)=c\bigl(a_1a_1P\bigr), \\
  c\bigl(W_1^2\bigr) & = & c\bigl(a_2a_1a_1a_2a_1a_1a_1a_2a_1a_1a_2a_1a_1a_1\bigr)=c\bigl(a_2a_1a_1a_2Q\bigr).
\end{array}
$$
\end{example}

\begin{defn}\label{d2}
We associate to each linked pair $(P,Q)\in\mathbb{LP}_2(\WW,\ma{Z})$ a cyclic word $\gamma(P,Q)=c\bigl(W_1Z_1\bigr)$, where $W_1$ and $Z_1$ are linear words defined as follows:
\begin{enumerate}
\item[(i)] For $(P,Q)$ a linked pair of type $(1)$ or $(2)$, $W_1$ is the representative of $\WW$ obtained by cutting $\WW$ immediately before $p_2$ and $Z_1$ is the representative of $\ma{Z}$ obtained by cutting $\ma{Z}$ immediately before $q_2$.
\item[(ii)] For $(P,Q)$ a linked pair of type $(3)$, $W_1$ is the linear subword of $\WW$ that starts right after the end of $Y$ and ends right before the first letter of $Y$, and $Z_1$ is the
subword of $\ma{Z}$ that starts right after the last letter of $\overline{Y}$ and ends right before the beginning of $\overline{Y}$. (Observe that $Y$ may not be a subword of $\WW$ or of $\ma{Z}$, but we can always find the first and last letters of $Y$ in $\WW$ and $\ma{Z}$.)
\end{enumerate}
\end{defn}

\begin{prop}
For each pair of reduced cyclic words $\WW$ and $\ma{Z}$, and for each linked pair $(P,Q)\in\mathbb{LP}_2(\WW,\ma{Z})$, the cyclic word $\gamma(P,Q)$ is a cyclic reduced word. In particular, $\gamma(P,Q)$ is non-empty.
\end{prop}
The proof of this result is in \cite{moira}.

\begin{lem}\label{l2}
Let $\mathcal{V}$ and $\m$ be reduced cyclic words and
$(P,Q)\in\mathbb{LP}_2(\mathcal{V},\m)$, then
$\gamma(P,Q)=\gamma(Q,P)$.
\end{lem}
\begin{proof}
For $(P,Q)\in\mathbb{LP}_2(\mathcal{V},\m)$ we have that $(Q,P)\in\mathbb{LP}_2(\m,\mathcal{V})$.\\
By definition $\gamma(P,Q)=c\bigl(V_1W_1\bigr)$ where $V_1$ and $W_1$ are
linear words defined in \ref{d2}. Subsequently, we have that
$\gamma(Q,P)=c\bigl(W_1V_1\bigr)$ and $\gamma(P,Q)=\gamma(Q,P)$.
\end{proof}

\begin{defn}
We define the bracket $[\;,\;]:\mathbb{V}\otimes\mathbb{V}\longrightarrow\mathbb{V}$ as the linear map such that for each pair of cyclic words, $\WW$ and $\ma{Z}$,
$$\bigl[\WW,\ma{Z}\bigr] = \sum_{(P,Q)\in\mathbb{LP}_2(\WW,\ma{Z})}sign(P,Q)\gamma(P,Q).$$

By Proposition \ref{p1}, $\mathbb{LP}_2(\WW,\ma{Z})$ is a finite set and as a consequence the bracket is well defined.
\end{defn}

In the next results we consider that the linked pairs are of the type
 $(2)$, for the type $(3)$ the results are similar.

\begin{prop}
For $(P,Q)\in\mb{LP}_2(\VV,\WW)$, with $l(\WW)\geq l(\VV)$, we have that
$Q\subset\WW^j$ for $j\leq 2$. Moreover, if $j=2$ and
$W_1=q_2B_1q_1B_2$, then  $\WW^2=c\bigl(q_2B_1q_1X\bigr)$.
\end{prop}
\begin{proof}
For $l(\WW)=l(\VV)$, by the remark \ref{r2}, we have that  $P\subset \VV$ and
$Q\subset \WW$. As a consequence $j=1$. For $l(\WW)>l(\VV)$, by the proposition
\ref{p2}, we conclude that $\displaystyle{j\leq 2+\frac{l(\VV)}{l(\WW)}<3}$.

Suppose now that $j=2$. Since $(P,Q)$ is a linked pair of type $(2)$, then we
have  linear representations of $\VV$ and $\WW$ as
$V_1=p_2A_1p_1A_2$ and $W_1=q_2B_1q_1B_2$ respectively. We know that
$Q\subset\WW^2$ and $Q\not\subset\WW$, then
$Q=q_1B_2q_2B_1q_1B_2q_2$, $X=B_2q_2B_1q_1B_2$ and
$W_1^2=q_2B_1q_1qX$.
\end{proof}

\begin{prop}\label{p2}
For $(P,Q)\in\mb{LP}_2(\VV,\WW)$, there exist
$\bigl(P_1,Q_1\bigr)\in\mb{LP}_2\bigl(\gamma(P,Q),\WW\bigr)$ and
$\bigl(P_2,Q_2\bigr)\in\mb{LP}_2(\VV,\gamma(P,Q))$ such that $P_1=p_1XW_1p_2$,
$Q_1=q_1XW_1q_2$, $P_2=p_1XV_1p_2$ and $Q_2=q_1XV_1q_2$.
\end{prop}
\begin{proof}
For $l(\VV)=l(\WW)$, we have that  $Q\subset\WW$, $P\subset \VV$, $V_1=p_2Ap_1X$ and $W_1=q_2Bq_1X$. Therefore, $\gamma(P,Q)=c\bigl(p_2Ap_1Xq_2Bq_1B\bigr)$ and $P_1=p_1XW_1p_2$,
$Q_2=q_1XV_1q_2$ are linear subwords of $\gamma(P,Q)^2$.  Note that $\WW^2=c\bigl(q_1XW_1q_2X\bigr)$ as a consequence of $Q_1$ is a linear subword of $\WW^2$.
By the other hand $\VV^2=c\bigl(p_1XV_1p_2A\bigr)$ hence $P_2$ is a linear sub-word of $\VV^2$.

For $l(\VV)<l(\WW)$, we are going to prove only the case $j=2$, the other case, when $j=1$, is left to the reader (the procedure is similar).\\
Let  $P=p_1Xp_2$, $Q=q_1Xq_2$ and $V_1=p_2A_1p_1A_2$, $W_1=q_2B_1q_1B_2$ be linear representatives of $\mathcal{V}$ and $\mathcal{W}$ respectively.\\
Applying the lemma \ref{l2} we have that $ W_1^2=q_2B_1q_1X$, where $ X=B_2q_2B_1q_1B_2$.\\
By the other hand, since $ P\in\mathcal{V}^i$, then we conclude that $ V_1^i=p_2A_1p_1X$ and $X=A_2V_1^{i-1}$, where $i\geq 2$. This is because $l(\ma{W})>l(\ma{V})$ and $j=2$.

In the figure \ref{fig5} we give a pictorial representation of these words.
\begin{figure}[htp]
    \centering
    \scalebox{1.2}{\includegraphics{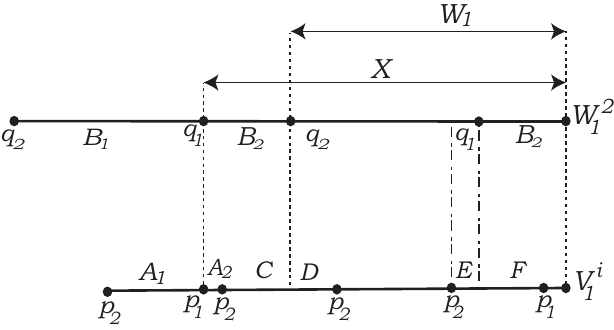}}
    \caption{Representation of the power of $\ma{W}$ and $\ma{V}$.}\label{fig5}
\end{figure}

Now, since $l(\mathcal{V})<l(\mathcal{W})$ and $l(X)<l(\VV)+l(\WW)$, then we have that $l(B_2)<l(\VV)$, $W_1=q_2DV_1^l=q_2DV_1^{l-1}p_2Eq_1B_2$ and $V_1=p_2Eq_1B_2=p_2Cq_2D$ (see figure \ref{fig5}). Moreover, $B_2=A_2p_2C=Fp_1A_2$.

Since $(P,Q)$ is a linked pair, then the pairs $(P_1,Q_1)$ and $(P_2,Q_2)$ are linked pairs too. To complete the proof we need to prove that the pairs live in the corresponding words.
\begin{enumerate}
\item $Q_1$ is a subword of some power of $\ma{W}$: this is because $W_1= q_2B_1q_1B_1$ and $X=B_2W_1$. Hence $W_1^3=q_2B_1q_1XW_1$. In particular $Q_1\subset c\bigl(W_1^3\bigr)$.
\item $P_1$ is a subword of some power of $\gamma(P,Q)$: first, note that $W_1V_1=q_2Dp_2CW_1$, where $V_1=p_2Cq_2D$. Then $\gamma(P,Q)^2=c\bigl(W_1V_1\bigr)^2=c\bigl(q_2Dp_2CW_1W_1V_1\bigr)=c\bigl(q_2Ap_1A_2p_2CW_1W_1V_1\bigr)=c\bigl(q_2Ap_1XW_1V_1\bigr)$. Therefore, $P_1\subset \gamma(P,Q)^2$.
\item $P_2$ is a subword of some power of $\ma{V}$: similar as in  the first case we have that  $V_1^i=p_2A_1p_1X$, hence $V_1^{i+1}=p_2A_1p_1XV_1$ and consequently $P_2\subset c\bigl(V_1^{i+1}\bigr)$.
\item $Q_2$ is a subword of some power of $\gamma(P,Q)$: first, note that $V_1^l=p_2Eq_1B_2$. Hence
$$\gamma(P,Q)^2 = c\bigl(p_2Eq_1B_2W_1V_1W_1\bigr) = c\bigr(p_2Eq_1XV_1W_1\bigl) = c\bigr(p_2Eq_1XV_1q_2B_1q_1B_2\bigl)$$
 and for that reason we have $Q_2\subset \gamma(P,Q)^2$.
\end{enumerate}
\end{proof}

\begin{cor}
For $(P,Q)\in\mb{LP}_2(\ma{V},\ma{W})$, with $P=p_1Xp_2$ and $Q=q_1Xq_2$, we have that $X\subset\gamma(P,Q)^2.$
\end{cor}

The next technical results are necessary to give a proof of the Jacobi axiom. They guarantee that given $(P,Q)\in\mb{LP}_2(\ma{V},\ma{W})$ and $(R,S)\in\mb{LP}_2\bigl(\ga(P,Q),\ma{Z}\bigr)$ there exist $\bigl(P_1,Q_1\bigr)$ and $\bigl(R_1,S_1\bigr)$ such that the term $\bigl[[\ma{V},\ma{W}]_{(P,Q)},\ma{Z}\bigr]_{(R,S)}$ is canceled by a term in $\bigl[[\ma{W},\ma{Z}],\ma{V}\bigr]$ or $\bigl[[\ma{Z},\ma{V}],\ma{W}\bigr]$ related with the linked pairs constructed.

\begin{lem}\label{l3}
For $\mathcal{V}$, $\ma{W}$ and $\ma{Z}$ cyclic words such that $l(\ma{Z})\leq l({\ma{V}})\leq l(\ma{W})$ and $(P,Q)\in \mb{LP}_2(\ma{V},\ma{W})$, $(R,S)\in\mb{LP}_2\bigl(\ma{Z},\ga(P,Q)\bigr)$, exists a linked pair $(T,U)\in\mb{LP}_2(\ma{W},\ma{Z})$ or $(T,U)\in\mb{LP}_2(\ma{V},\ma{Z})$.
\end{lem}
\begin{proof}
We are considering $P=p_1Xp_2$, $Q=q_1Xq_2$, $R=r_iYr_2$ and $S=s_1Ys_2$. Remember that, by lemma \ref{l2}, we have that $l(Y)<l(\VV)+l(\WW)+l(\ma{Z})$. This is because $Y$ is a subword of a power of $\ma{Z}$ and $\gamma(P,Q)$.

We can suppose that $s_2\in V_1$ (for $s_2\notin V_1$ the procedure is analogous). Note that since $l\bigl(\gamma(P,Q)\bigr)>l(\ma{Z})$, then $S\subset\gamma(P,Q)^2$.
\begin{figure}[htp]
    \centering
    \scalebox{1.2}{\includegraphics{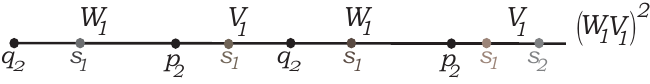}}
    \caption{Possible position of the point $s_1\in\gamma(P,Q)^2$.}\label{fig6}
\end{figure}
In figure \ref{fig6} we can see the posible positions of the point $s_1$.\\Presently we study the four possibilities.
\begin{enumerate}
\item For $s_1\in V_1$, the same representative where $s_2$ lives, we have that $(R,S)\in\mb{LP}_2(\ma{Z},\ma{V})$. Then we define $T=S$ and  $U=R$.
\item For $s_1\in W_1$, note that $XW_1\not\subset Y$. In this case $Y$ intersects $W_1$.
\begin{enumerate}
\item If $sign(P,Q)\neq sign(R,S)$, then we can construct a linked pair $(T,U)\in\mb{LP}_2(\ma{W},\ma{Z})$. Since $s_1\in W_1$, we can find $Z\subset W_1$ such that $W_1=Z_1s_1Z$. For this reason we define $T=s_1Zq_2$ and $U=r_1Zp_2$.
\item If $sign(P,Q) = sign(R,S)$, we have two possible cases. If $s_1\in XV_1$, then $(R,S)\in\mb{LP}_2(\VV,\ma{Z})$ but if $s_1\not\in XV_1$, then it is possible to construct $(T,U)\in\mb{LP}_2(\VV,\ma{Z})$. Since $s_2\in V_1$, then we can decompose $V_1=p_2Bs_2C$ and we can define $T=p_1Xp_2Bs_2$ and $U=q_1Xp_2Br_2$.
\end{enumerate}
\item For $s_1\in V_1$ or $s_1\in W_1$, similarly as before, we need to divide the discussion in subcases.
\begin{enumerate}
\item Suppose $XW_1\subset Y$. In this case $(P_1,Q_1)\in\mb{LP}_2\bigl(\gamma(P,Q),\ma{W}\bigr)$ is a posible solution. This because $P_1=p_1XW_1p_2\subset\gamma(P,Q)^2$ and $XW_1\subset Y$. Then $T=Q_1$, $U=P_1$ and $(T,U)\in\mb{LP}_2(\WW,\ma{Z})$.
\item Suppose that $XW_1\not\subset Y$. If $s_1\in V_1$ and $sign(P,Q)\neq sign(R,S)$, then we define $T=s_1X_1W_1q_2$ and $U=r_1X_1W_1p_2$, where $V_1=X_2s_1X_2$ and $(T,U)\in\mb{LP}_2(\WW,\ma{Z})$.\\ If $sign(P,Q)= sign(R,S)$ and $s_1\in XV_1$, then $(R,S)\in\mb{LP}_2(\VV,\ma{Z})$. But if $s_1\not\in XV_1$, then $(T,U)\in\mb{LP}_2(\VV,\ma{Z})$, where $T=p_1Xp_2Bs_2$ and $U=q_1Xp_2Br_2$. Finally, if $s_1\in W_1$ and $sign(P,Q)\neq sign(R,S)$, then $(T,U)\in\mb{LP}_2(\VV,\ma{Z})$, where $T=s_1Xp_2$ and $U=r_1Xq_2$.
\end{enumerate}
\end{enumerate}
\end{proof}

\begin{lem} \label{l4}
For $\mathcal{V}$, $\ma{W}$ and $\ma{Z}$ cyclic words such that $l(\ma{Z})\leq l({\ma{V}})\leq l(\ma{W})$ and $(P,Q)\in \mb{LP}_2(\ma{V},\ma{W})$, $(R,S)\in\mb{LP}_2\bigl(\ma{Z},\ga(P,Q)\bigr)$, then we have the following consequences:
\begin{enumerate}
\item[(a)] For $(T,U)\in\mb{LP}_2(\ma{W},\ma{Z})$ there exists $\bigl(T',U'\bigr)\in\mb{LP}_2\bigl(\ma{V},\ga(T,U)\bigr)$ such that
$$\begin{array}{ccc}
    \ga_{(T,U)}\bigl(T',U'\bigr) & = & \ga_{(P,Q)}(R,S) \\
    &\textrm{and}&\\
    \bigl[\ma{V},[\ma{W},\ma{Z}]_{(T,U)}\bigr]_{(T',U')} & = & -\bigl[\ma{Z},[\ma{V},\ma{W}]_{(P,Q)}\bigr]_{(R,S)}.
  \end{array}
$$
\item[(b)] For $(T,U)\in\mb{LP}_2(\ma{V},\ma{Z})$ there exists $\bigl(T',U'\bigr)\in\mb{LP}_2\bigl(\ma{W},\ga(T,U)\bigr)$ such that
$$\begin{array}{ccc}
\ga_{(T,U)}\bigl(T',U'\bigr) & = & \ga_{(P,Q)}(R,S)\\
& \textrm{and}&\\
\bigl[\ma{W},[\ma{Z},\ma{V}]_{(U,T)}\bigr]_{\bigl(T',U'\bigr)} & = & -\bigl[\ma{Z},[\ma{V},\ma{W}]_{(P,Q)}\bigr]_{(R,S)}.
\end{array}$$
\end{enumerate}
\end{lem}
\begin{proof}
For $(P,Q)\in\mb{LP}_2(\ma{V},\ma{W})$ and $(R,S)\in\mb{LP}_2\bigl(\ma{Z},\ga(P,Q)\bigr)$, we need to calculate $$\bigl[\ma{Z},[\ma{V},\ma{W}]_{(P,Q)}\bigr]_{(R,S)}=sign(P,Q)sign(R,S)\ga_{(P,Q)}(R,S).$$

By definition $\ga(P,Q)=c\bigl(V_1W_1\bigr)$, where $V_1$ and $W_1$ are linear representatives of $\ma{V}$ and $\ma{W}$ respectively. By the other hand, as $(R,S)\in\mb{LP}_2\bigl(\ma{Z},\ga(P,Q)\bigl)$ is
$\ga_{(P,Q)}(R,S)=c\bigl(Z_2(V_1W_1)_2\bigr)$, where $Z_2$ is a linear representative of $\ma{Z}$ starting in $r_2$ and $(V_1W_1)_2$ is a linear representative of $\ga(P,Q)$ starting in $s_2$.\\

We suppose that $s_2\in V_1$ (the other case is an exercise for the reader). As a consequence of the assumption we have that $V_1=p_2Bs_2A$. In this case we have two possibilities with respect to $\ma{Z}$:
\begin{enumerate}
\item[(i)] The word $Z_2$ can be expressed as $Z_2=r_2Cp_2B$. In this case we have that $\bigl(V_1W_1\bigr)_2=s_2AW_1p_2B$ and
$$\ga_{(P,Q)}(R,S)=c\bigl(r_2Cp_2Bs_2AW_1p_2B\bigr)=c\bigl(p_2Bs_2AW_1p_2Br_2C\bigr)=c\bigl(V_1W_1Z_1\bigr).$$
\item[(ii)] The word $p_2B$ can be expressed as $p_2B=Z_1^ip_2D$ and $Z_1=p_2Dr_2C$,  $Z_2=r_2Cp_2D$. As a consequence
$$\begin{array}{ccccc}
    c\bigl(Z_2(V_1W_1)_2\bigr) & = & \bigl(r_2Cp_2Ds_2AW_1p_2B\bigr)     & = & \bigl(r_2Cp_2Ds_2AW_1Z_1^ip_2D\bigr) \\
                               & = &   c\bigl(Z_1^{i+1}p_2Ds_2AW_1\bigr) & = & c\bigl(Z_1p_2Bs_2AW_1\bigr)\\
                               & = & c\bigl(V_1W_1Z_1\bigr).              &   &
  \end{array}
$$
\end{enumerate}
Finally, we have that  $\ga_{(P,Q)}(R,S)=c\bigl(V_1W_1Z_1\bigr)$.\\

We prove the result for some cases, the remaining cases are left.
\begin{enumerate}
\item[(1)]  For $XW_1\subset Y$, there exists $(T,U)\in\mb{LP}_2(\ma{W},\ma{Z})$, where $T=q_1X_1q_2$ and $U=p_1X_1p_2$ ($X_1=XW_1$).\\
By the proposition \ref{p2} we have that there exists  $\bigl(U_2,T_2\bigr)\in\mb{LP}_2\bigl(\ma{Z},\ga(T,U)\bigr)$ where $U_2=p_1X_1Z_1p_2$  and $T_2=q_1X_1Z_1q_2$. As $T_2$ is a subword of some power of $\ga(T,U)$ and $Q\subset T_2$ is $Q$ a subword of some power of $\ga(T,U)$, then $(P,Q)\in\mb{LP}_2\bigl(\ma{V},\ga(T,U)\bigr)$. In this case we define $T'=P$ and  $U'=Q$. Moreover,
$$\ga_{(T,U)}\bigl(T',U'\bigr)=c\bigl(V_1(Z_1W_1)_1\bigr)=c\bigl(V_1W_1Z_1\bigr)=\ga_{(P,Q)}(R,S).$$
Now, we determine the term  $\bigl[\ma{V},[\ma{W},\ma{Z}]_{(T,U)}\bigr]_{\bigl(T',U'\bigr)}$.
\begin{alignat*}{2}
\bigl[\ma{V},[\ma{W},\ma{Z}]_{(T,U)}\bigr]_{\bigl(T',U'\bigr)}&= sign\bigl(T',U'\bigr)sign\bigl(T,U\bigr)\ga_{(T,U)}(T',U')\\
&=sign(P,Q)sign(Q,P)\ga_{(P,Q)}(R,S)\\
&=-\ga_{(P,Q)}(R,S).
\end{alignat*}
Note that $sign(P,Q)=sign(R,S)$, hence $\bigl[\ma{V},[\ma{W},\ma{Z}]_{(T,U)}\bigr]_{\bigl(T',U'\bigr)}=-\bigl[\ma{Z},[\ma{V},\ma{W}]_{(P,Q)}\bigr]_{(R,S)}.$ But, if $sign(P,Q)\neq sign(R,S)$, then this is not true.
\item[(2)]  For $XW_1\subset Y$ and $sign(P,Q)\neq sign(R,S)$, we consider $T=p_1X_2Bs_2$ and $U=q_1Xp_2Br_2$, with $(T,U)\in\mb{LP}_2(\VV,\ma{Z})$ and $T'=P$, $U'=Q$, where is easy to prove that $(P,Q)\in\mb{LP}_2\bigl(\gamma(T,U),\WW\bigr)$.

We can write $V_2=s_2Ap_2B$ and $Z_2=r_2Cp_2B$, then $V_2Z_2=s_2Ap_2Br_2Cp_2B$, and $\bigl(V_2Z_2\bigr)_1=p_2Br_2Cp_2Bs_2A=Z_1V_1$. Moreover, $\ga_{(T,U)}\bigl(T',U'\bigr)=c\bigl((V_2Z_2)_1W_1\bigr)=c\bigl(Z_1V_1W_1\bigr)=c\bigr(V_1W_1Z_1\bigl).$ Consequently,
$$\begin{array}{ccc}
\bigl[\ma{W},[\ma{Z},\ma{V}]_{(U,T)}\bigr]_{\bigl(U',T'\bigr)} & = & sign\bigl(U',T'\bigr)sign(U,T)\ga_{(U,T)}\bigl(U',T'\bigr)\\
                                                     & = & sign(Q,P)sign(Q,P)\ga_{(U,T)}\bigl(U',T'\bigr)\\
                                                     & = & \ga_{(P,Q)}(R,S)\\
                                                     & = & -\bigl[\ma{Z},[\ma{V},\ma{W}]_{(P,Q)}\bigr]_{(R,S)}.
\end{array}$$
\item[(3)] Let now $XW_1\not\subset Y$ and $sign(P,Q)\neq sign(R,S)$, by lemma \ref{l2}, there exists $(T,U)\in \mb{LP}_2(\ma{W},\ma{Z})$, with $T=r_1X_2q_2$ and $U=s_1X_2p_2$. We need to guaranty the  existence of another linked pair $\bigl(T',U'\bigr)\in\mb{LP}_2\bigl(\ma{V},\ga(T,U)\bigr)$. \\
    Suppose the case $s_1\not\in XV_1$. In this situation, we can construct, applying the proposition \ref{p2}, the linked pair $(R_2,S_2)\in\mb{LP}_2\bigl(\ma{W},\ga(T,U)\bigr)$, where $R_2=r_1X_2W_1q_2$ and $S_2=s_1X_2W_1p_2$.

Note that $Q\subset S_2$, hence $(P,Q)\in\mb{LP}_2\bigl(\ma{V},\ga(R_1,S_1)\bigr)$. Therefore, we can define $T'=P$ and $U'=Q$\\
As $\bigl(T',U'\bigr)\in\mb{LP}_2\bigl(\ma{V},\ga(T,U)\bigr)$ we have that $\ga_{(T,U)}(T',U')=c\bigl(V_1(W_1Z_1)_1\bigr)=c\bigl(V_1W_1Z_1\bigr)$. Then $\ga_{(P,Q)}(R,S)=\ga_{(T,U)}(T',U')$. Finally, we can calculate the following term $\bigl[\ma{V},[\ma{W},\ma{Z}]_{(T,U)}\bigr]_{\bigl(T',U'\bigr)}$.
$$\begin{array}{ccc}
\bigl[\ma{V},[\ma{W},\ma{Z}]_{(T,U)}\bigr]_{\bigl(T',U'\bigr)} & = & sign\bigl(T',U'\bigr)sign(T,U)\ga_{(T,U)}\bigl(T',U'\bigr)\\
                                                     & = & sign(P,Q)sign(R,S)\ga_{(T,U)}\bigl(T',U'\bigr)\\
                                                     & = & -\ga_{(P,Q)}(R,S)\\
                                                     & = & -\bigl[\ma{Z},[\ma{V},\ma{W}]_{(P,Q)}\bigr]_{(R,S)}.
\end{array}$$
\end{enumerate}
\end{proof}

\begin{thm}\label{t2}
$(\mathbb{V},[\;,\;])$ is a Lie algebra.
\end{thm}
\begin{proof}
\begin{enumerate}
\item[(1)] Skew-symmetry condition: $[\,,\,]\circ s=-[\,,\,]$.
\begin{alignat*}{2}
  [\,,\,]\circ s (\VV\ot\WW) & = [\WW,\VV] \\
   & =  \sum_{(P,Q)\in\mathbb{LP}_2(\WW,\VV)}sign(P,Q)\gamma(P,Q) \\
   & =  -\sum_{(Q,P)\in\mathbb{LP}_2(\VV,\WW)}sign(Q,P)\gamma(Q,P) \\
   & =  -[\WW,\VV].
\end{alignat*}

\item[(2)] Jacobi identity:
$[\,,\,]\circ\bigl(id\otimes[\,,\,]\bigr)\circ\bigl(id+\varepsilon+\varepsilon^2\bigr)=0$.
$$\bigl([\,,\,]\circ(id\otimes[\,,\,])\bigr)\circ\bigl(id+\varepsilon+\varepsilon^2\bigr)\bigl(\VV\ot\WW\ot\ma{Z}\bigr)=\bigl[\VV,[\WW,\ma{Z}]\bigr]+\bigl[\WW,[\ma{Z},\VV]\bigr]+\bigl[\ma{Z},[\VV,\WW]\bigr],$$
and
\begin{alignat*}{2}
\bigl[\ma{Z},[\VV, \ma{W}]\bigr]&=\sum_{(P,Q)\in\mb{LP}_2(\VV,\ma{W})}sign(P,Q)\bigl[\ma{Z},\gamma(P,Q)\bigr]\\
&=\sum_{(P,Q)\in\mb{LP}_2\bigl(\VV,\ma{W}\bigr)}\sum_{(R,S)\in\mb{LP}_2\bigl(\ma{Z},\gamma(P,Q)\bigr)}sign(P,Q)sign(R,S)\gamma(R,S)_{\gamma(P,Q)}.
\end{alignat*}
Notation: $\bigl[\ma{Z},[\ma{V},\ma{W}]_{(P,Q)}\bigr]_{(R,S)}=sign(P,Q)sign(R,S)\gamma(R,S)_{\gamma(P,Q)}.$\\

We can suppose, without any lost of generality, that $l(\ma{Z})\leq l(\ma{V})\leq l(\ma{W})$.

For $(P,Q)\in\mb{LP}_2(\ma{V,\ma{}W})$ and $(R,S)\in\mb{LP}_2\bigl(\ma{Z},\ga(P,Q)\bigr)$, by the lemmas \ref{l3} and \ref{l4}, there exist $(T,U)\in\mb{LP}_2(\ma{V},\ma{Z})$ and
$\bigl(T',U'\bigr)\in\mb{LP}_2\bigl(\ma{W},\ga(T,U)\bigr)$ such that
$$\bigl[\ma{W},[\ma{Z},\ma{V}]_{(U,T)}\bigr]_{\bigl(T',U'\bigr)} =-\bigl[\ma{Z},[\ma{V},\ma{W}]_{(P,Q)}\bigr]_{(R,S)}$$ or there exist $(T,U)\in\mb{LP}_2(\ma{W},\ma{Z})$ and
$\bigl(T',U'\bigr)\in\mb{LP}_2\bigl(\ma{V},\ga(T,U)\bigr)$ such that
$$\bigl[\ma{V},[\ma{W},\ma{Z}]_{(T,U)}\bigr]_{\bigl(T',U'\bigr)}=-\bigl[\ma{Z},[\ma{V},\ma{W}]_{(P,Q)}\bigr]_{(R,S)}.$$
Then, for each pair $(P,Q)\in\mb{LP}_2(\ma{V},\ma{W})$ and for each pair $(R,S)\in\mb{LP}_2\bigl(\ma{Z},\ga(P,Q)\bigr)$, there exist another pair of linked pairs such that $\bigl[\ma{Z},[\ma{V},\ma{W}]_{(P,Q)}\bigr]_{(R,S)}$ is canceled by the term corresponding to these new pairs. This imply that $\bigl[\mathcal{Z},[\mathcal{V},\m]\bigr]+\bigl[\m,[\mathcal{Z},\mathcal{V}]\bigr]+\bigl[\mathcal{V}, [\m,\mathcal{Z}]\bigr]=0$.

\end{enumerate}
\end{proof}

\section{The compatibility between these structures}
\begin{lem}\label{l1}
For $(P,Q)\in\mb{LP}_1(\VV)$ and $(R,S)\in\mb{LP}_2\bigl(\de_1(P,Q),\de_2(P,Q)\bigr)$, there exist $\bigl(R_1,S_1\bigr)\in\mb{LP}_1(\VV)$  such that $(Q,P)\in\mb{LP}_2\bigl(\de_1(R_1,S_1),\de_2(R_1,S_1)\bigr)$ with
$$\begin{array}{rcl}
    sign\bigl(R_1,S_1\bigr) & = & sign(R,S) \\
    &\textrm{and}&\\
    \gamma_{(P,Q)}(R,S) & = & \gamma_{\bigl(R_1,S_1\bigr)}(Q,P)
  \end{array}
$$
\end{lem}
\begin{proof}
We construct the linked pair in some cases, the other cases are left to the reader.\\
We can suppose that $l\bigl(\de_2(P,Q)\bigr)\geq l\bigl(\de_1(P,Q)\bigr)$ (the other case is analogous).
\begin{enumerate}
\item For $l\bigl(\de_2(P,Q)\bigr)= l\bigl(\de_1(P,Q)\bigr)$, since we have that $R\subset \de_1(P,Q)$ and $S\subset\de_2(P,Q)$, then $(R,S)\in\mb{LP}_1(\VV)$.

Remember that $V_1$ is the linear representative of $\de_1(P,Q)$ starting at $p_2$. As a consequence we can suppose that $$V_1=p_2B_1r_1Y_1q_1Y_2r_2X_2,$$
where $Y=Y_1q_1Y_2$ and $X=Y_2r_2X_2$. In a similar form we can considere that $$V_2=q_2A_1s_1Ys_2A_2p_1X.$$
Therefore  $V_1V_2=p_2B_1r_1Y_1q_1Y_2r_2X_2q_2A_1s_1Ys_2A_2p_1X$ and
$$\begin{array}{rcl}
    \de_1(R,S) & = & c\bigl(r_2X_2q_2A_1s_1Y_1q_1Y_2\bigr)=c\bigl(A_1s_1Y_1q_1Xq_2\bigr)\ni Q. \\
    \de_2(R,S) & = & c\bigl(s_2A_2p_1Xp_2B_1r_1Y\bigr) \ni P.
  \end{array}
$$ Hence $(Q,P)\in\mb{LP}_2\bigl(\de_1(R,S),\de_2(R,S)\bigr)$.

Finally, we calculate $\gamma_{(P,Q)}(R,S)$ and $\gamma_{(R,S)}(Q,P)$.
$$\begin{array}{lclcl}
    \gamma_{(P,Q)}(R,S) & = & c\bigl(r_2X_2p_2B_1r_1Ys_2A_2p_1Xq_2A_1s_1Y\bigr) & = & c\bigl(p_2B_1r_1Ys_2A_2p_1Xq_2A_1s_1Yr_2X_2\bigr)\\
     & = & c\bigl(p_2B_1r_1Ys_2A_2p_1Xq_2A_1s_1Y_1q_1X\bigr) &&
  \end{array}
$$
and
$$\begin{array}{rclcl}
    \gamma_{(R,S)}(Q,P) & = & c\bigl(q_2A_1s_1Y_1q_1Xp_2B_1r_1Ys_2A_2p_1X\bigr) & = & c\bigl(p_2B_1r_1Ys_2A_2p_1Xq_2A_1s_1Y_1q_1X\bigr),
  \end{array}
$$
and we conclude that $\gamma_{(P,Q)}(R,S)=\gamma_{(R,S)}(Q,P).$
\item For $l\bigl(\de_2(P,Q)\bigr) > l\bigl(\de_1(P,Q)\bigr)$ we consider the case $j=1$ and $sign(P,Q)\neq sign(R,S)$.

Since $j=1$ we have $S\subset \de_2(P,Q)$, and as a consequence $V_2=q_2A_1s_1Ys_2A_2p_1X$ and $V_1=p_2B_1r_1C_2r_2B_2q_1X$, where $Y=C_2r_2B_2q_1Xp_2B_1r_1C_2$. In this case we have
$$V_1V_2=p_2B_1r_1C_2r_2B_2q_1Xq_2A_1s_1C_2r_2B_2q_1Xp_2B_1r_1C_2s_2A_2p_1X$$
and $$\begin{array}{ccc}
        R_1 & = & r_1C_2r_2B_2q_1Xq_2, \\
        S_1 & = & s_1C_2r_2B_2q_1Xp_2,
      \end{array}
$$ is a linked pair of $\VV$ and $$\begin{array}{rcl}
                                     \de_1(R_1, S_1) & = & c\bigl(q_2A_1s_1C_2r_2B_2q_1X\bigr)\ni Q. \\
                                     \de_2(R_1, S_1) & = & c\bigl(p_2B_1r_1C_2s_2A_2p_1Xp_2B_1r_1C_2r_2B_2q_1X\bigr) \ni P.
                                   \end{array}
$$
Finally, we calculate $\gamma_{(P,Q)}(R,S)$ and $\gamma_{\bigl(R_1,S_1\bigr)}(Q,P)$.
$$\begin{array}{lcl}
    \gamma_{(P,Q)}(R,S) & = & c\bigl(r_2B_2q_1Xp_2B_1r_1C_2s_2A_2p_1Xq_2A_1s_1Y\bigr),\\
    \gamma_{\bigl(R_1,S_1\bigr)}(Q,P) & = & c\bigl(q_2A_1s_1Yr_2B_2q_1Xp_2B_1r_1C_2s_2A_2p_1X\bigr),
  \end{array}
$$ and we conclude that $\gamma_{(P,Q)}(R,S)=\gamma_{\bigl(R_1,S_1\bigr)}(Q,P).$
\item For $l\bigl(\de_2(P,Q)\bigr) > l\bigl(\de_1(P,Q)\bigr)$ we consider the case $j=2$ and $sign(P,Q) = sign(R,S)$.

This case is a little more complicated, but with some careful it is posible to construct a linear representative of $\de_1(P,Q)$ and $\de_2(P,Q)$ as follows
$$V_1=p_2B_4s_1D_2r_1C_2r_2D_1s_2B_3q_1X\quad\textrm{and}\quad V_2=q_2A_1s_1E_2s_2A_2p_1X,$$
where $E_2=D_2r_1C_2=C_2r_2D_1$. Then
$$V_1V_2=p_2B_4s_1E_2r_2D_1s_2B_3q_1Xq_2A_1s_1E_2s_2A_2p_1X,$$
with  $A_1=H_2r_2D_4=H_2r_2D_1s_2B_3q_1Xp_2B_4$.
Therefore, we can find a new  linked pair in $\VV$:
$$\begin{array}{ccc}
    R_1 & = & p_1Xp_2B_4s_1E_2r_2, \\
    S_1 & = & q_1Xp_2B_4s_1E_2s_2,
  \end{array}
$$ and
$$\begin{array}{rcl}
    \de_1\bigl(R_1, S_1\bigr) & = & c\bigl(r_2D_1s_2B_3q_1Xq_2H_2r_2D_1s_2B_3q_1Xp_2B_4s_1E_2\bigr)\ni Q. \\
    \de_2\bigl(R_1, S_1\bigr) & = & c\bigl(s_2A_2p_1Xp_2B_4s_1E_2\bigr)\ni P.
  \end{array}
$$
As before, it rests to determine $\gamma_{(P,Q)}(R,S)$ and $\gamma_{\bigl(R_1,S_1\bigr)}(Q,P)$:
$$\begin{array}{lclcl}
    \gamma_{\bigl(R_1,S_1\bigr)}(Q,P) & = & c\bigl(q_2H_2r_2D_4s_1E_2r_2D_4s_1E_2s_2A_2p_1X\bigr) & = & c\bigl(q_2A_1s_1E_2r_2D_4s_1E_2s_2A_2p_1X\bigr),\\
     \gamma_{(P,Q)}(R,S)& = & c\bigl(r_2D_4s_1D_2r_1C_2s_2A_2p_1Xq_2A_1s_1E_2\bigr) & = & c\bigl(q_2A_1s_1E_2r_2D_4s_1E_2s_2A_2p_1X\bigr).
  \end{array}$$
\end{enumerate}
\end{proof}

\begin{cor}\label{c1}
With the same conditions of the last lemma we have that: $$[\;,\,]_{(R,S)}\bigl(\de_{(P,Q)}(\VV)\bigr)=-[\;,\,]_{(Q,P)}\bigl(\de_{\bigl(R_1,S_1\bigr)}(\VV)\bigr).$$
\end{cor}
\begin{proof}
Note that $$[\;,\,]_{(R,S)}\bigl(\de_{(P,Q)}(\VV)\bigr)=sign(P,Q)sign(R,S)\gamma_{(P,Q)}(R,S)$$ and $$[\; ,\,]_{(Q,P)}\bigl(\de_{\bigl(R_1,S_1\bigr)}(\VV)\bigr)=sign(Q,P)sign\bigl(R_1,S_1\bigr)\gamma_{(R_1,S_1)}(P,Q)=-sign(P,Q)sign\bigl(R_1,S_1\bigr)\gamma_{\bigl(R_1,S_1\bigr)}(P,Q)$$ and, by the lemma \ref{l1}, this is the same that $$-sign(P,Q)sign(R,S)\gamma_{(P,Q)}(R,S)=-[\;,\,]_{(R,S)}\bigl(\de_{(P,Q)}(\VV)\bigr).$$
\end{proof}

\begin{thm}
$(\mathbb{V},[\;,\;],\de)$ is an involutive Lie bialgebra.
\end{thm}
\begin{proof}
\begin{enumerate}
\item[(1)] Involutive: $[\;,\,]\circ\delta=0$.
$$
[\;,\,]\circ\de(\VV)=\sum_{(P,Q)\in\mb{LP}_1(\VV)}sign(P,Q)\bigl[\de_1(P,Q),\de_2(P,Q)\bigr].
$$
$$\bigl[\de_1(P,Q),\de_2(P,Q)\bigr]=\sum_{(R,S)\in\mb{LP}_2\bigl(\de_1(P,Q),\de_2(P,Q)\bigr)}sign(P,Q)sign(R,S)\gamma(R,S).$$
Consequently, we have that
$$[\;,\,]\bigl(\de(\VV)\bigr)=\sum_{(P,Q)\in\mb{LP}_1\bigl(\VV\bigr)}\sum_{(R,S)\in\mb{LP}_2\bigl(\de_1(P,Q),\de_2(P,Q)\bigr)}sign(P,Q)sign(R,S)\gamma(R,S)$$
To prove this result we need to guarantee that the term
$sign(P,Q)sign(R,S)\gamma(R,S)$ is canceled by another term in the
sum. This statement is a consequence of the lemma \ref{l1} and the corollary \ref{c1}.
\item[(2)] Compatibility condition: $\de\bigl([\VV,\WW]\bigr)=\VV\cdot
\de(\WW)-\WW\cdot\de(\VV)$, where the action is given by
$\ma{Z}\cdot(\VV\otimes \WW)=[\ma{Z},\VV]\otimes \WW+\VV\otimes
[\ma{Z},\WW]$.
\begin{alignat*}{2}
\de\bigl([\VV,\WW]\bigr)&=\de\left(\sum_{(P,Q)\in\mb{LP}_2(\VV,\WW)}sign(P,Q)\gamma(P,Q)\right)\\
&=\sum_{(P,Q)\in\mb{LP}_2\bigl(\VV,\WW\bigr)}\sum_{(R,S)\in\mb{LP}_1\bigl(\gamma(P,Q)\bigr)}sign(P,Q)sign(R,S)\de_1(R,S)\otimes\de_2(R,S).
\end{alignat*}
\begin{alignat*}{2}
\VV\cdot
\de(\WW)&=\sum_{(P,Q)\in\mb{LP}_1\bigl(\WW\bigr)}\sum_{(R,S)\in\mb{LP}_2\bigl(\VV,\de_1(P,Q)\bigr)}signs(P,Q)sign(R,S)\gamma(R,S)\otimes\de_2(P,Q)\\
&+\sum_{(P,Q)\in\mb{LP}_1\bigl(\WW\bigr)}\sum_{(R,S)\in\mb{LP}_2\bigl(\VV,\de_2(P,Q)\bigr)}signs(P,Q)sign(R,S)\de_1(P,Q)\otimes\gamma(R,S).\\
\WW\cdot\de(\VV)&=\sum_{(P,Q)\in\mb{LP}_1\bigl(\VV\bigr)}\sum_{(R,S)\in\mb{LP}_2\bigl(\WW,\de_1(P,Q)\bigr)}signs(P,Q)sign(R,S)\gamma(R,S)\otimes\de_2(P,Q)\\
&+\sum_{(P,Q)\in\mb{LP}_1\bigl(\VV\bigr)}\sum_{(R,S)\in\mb{LP}_2\bigl(\WW,\de_2(P,Q)\bigr)}signs(P,Q)sign(R,S)\de_1(P,Q)\otimes\gamma(R,S).
\end{alignat*}
\end{enumerate}
First, note that
\begin{equation}\label{e1}
\begin{array}{ccl}
 \de\bigl([\VV,\WW]\bigr) & = \displaystyle{\sum_{{\scriptscriptstyle (P,Q)\in\mb{LP}_2(\VV,\WW),\{(R,S),(S,R)\}\subset\mb{LP}_1(\gamma(P,Q))}}} & sg(P,Q)sg(R,S)\bigl\langle\de_1(R,S)\otimes\de_2(R,S) \\
   &  & -\de_2(R,S)\otimes\de_1(R,S)\bigr\rangle.
\end{array}
\end{equation}

If we have $(R,S)\in\mb{LP}_1\bigl(\gamma(P,Q)\bigr)$, then there are three different possible situation: $(R,S)\in\mb{LP}_1(\VV)$, $(R,S)\in\mb{LP}_1(\WW)$ or $(R,S)\in\mb{LP}_2\bigl(\VV,\WW\bigr)\backslash(P,Q)$.

In the next paragraph we illustrate the proof of the compatibility condition.  We suppose that $l(\VV)=l(\WW)$, and as a consequence we have that $P\subset \VV$ and $Q\subset\WW$. \\
$\bullet$ For $(R,S)\in\mb{LP}_1(\VV)$ we suppose that $V_1=p_2A_1r_1Yr_2A_2s_1A_3p_1X_1s_2X_2$ and $W_1=q_2Bq_1X$, where $X=X_1s_2X_2$ and $Y=A_3p_1X_1$.

First, we determine the following $$\begin{array}{ccl}
                        \de_{(R,S)}\bigl([\VV,\WW]_{(P,Q)}\bigr) & = & sign(P,Q)sign(R,S)\bigl\langle\de_1(R,S)\otimes\de_2(R,S)-\de_2(R,S)\otimes\de_1(R,S)\bigr\rangle.
 \end{array}$$
Now, since $\gamma(P,Q)=c\bigl(p_2A_1r_1Yr_2A_2s_1A_3p_1X_1s_2X_2q_2Bq_1X\bigr)$, then $\de_1(R,S)=c\bigl(r_2A_2s_1A_3p_1X_1\bigr)$ and $\de_2(R,S)=c\bigl(s_2X_2q_2Bq_1Xp_2A_1r_1Y\bigr)$.\\
Since $(R,S)\in\mb{LP}_1(\VV)$, we have $\de_1(R,S)_{\VV}=c\bigl(r_2A_2s_1A_3p_1X_1\bigr)$ and $\de_2(R,S)_{\VV}=c\bigl(s_2X_2p_2A_1r_1Y\bigr)=c\bigl(s_2X_2p_2A_1r_1A_3p_1X_1\bigr)=c\bigl(A_1r_1A_3p_1X_1s_2X_2p_2\bigr)=c\bigl(A_1r_1A_3p_1Xp_2\bigr)\ni P$.\\ We conclude that $(Q,P)\in\mb{LP}_2\bigl(\WW,\de_2(R,S)_{\VV}\bigr)$, and as a result we have associated to these pairs the term $\WW\cdot_{(Q,P)}\de(\VV)_{(R,S)}$.
$$\begin{array}{ccl}
                       \WW\cdot_{(Q,P)}\de(\VV)_{(R,S)}  & = & sign(Q,P)sign(R,S)\bigl\langle\de_1(R,S)\otimes\gamma(Q,P)-\gamma(Q,P)\otimes\de_1(R,S)\bigr\rangle \\
                         & = & -sign(P,Q)sign(R,S)\bigl\langle c\bigl(r_2A_2s_1A_3p_1X_1\bigr)\otimes c\bigl(q_2Bq_1Xp_2A_1r_1A_3p_1X\bigr)\\
                         && + c\bigl(q_2Bq_1Xp_2A_1r_1A_3p_1X\bigr)\otimes c\bigl(r_2A_2s_1A_3p_1X_1\bigr)\bigr\rangle\\
                         & = & -sign(P,Q)sign(R,S)\bigl\langle c\bigl(r_2A_2s_1A_3p_1X_1\bigr)\otimes c\bigl(q_2Bq_1Xp_2A_1r_1Ys_2X_2\bigr)\\
                         && + c\bigl(q_2Bq_1Xp_2A_1r_1Ys_2X_2\bigr)\otimes c\bigl(r_2A_2s_1A_3p_1X_1\bigr)\bigr\rangle.
                      \end{array}$$
Therefore $\de_{(R,S)}\bigl([\VV,\WW]_{(P,Q)}\bigr)=- \WW\cdot_{(Q,P)}\de(\VV)_{(R,S)}$.\\
$\bullet$ For $(R,S)\in\mb{LP}_1(\WW)$ we suppose that $W_1=q_2B_1r_1Yr_2B_2s_1B_3q_1X_1s_2X_2$ and $V_1=p_2Ap_1X$, where $X=X_1s_2X_2$ and $Y=B_3q_1X_1$.

In this case we have that $\de_1(R,S)_{\WW}=c\bigl(r_2B_2s_1B_3q_1X_1\bigr)$ and $\de_2(R,S)_{\WW}=c\bigl(s_2X_2q_2B_1r_1Y\bigr)=c\bigl(s_2X_2q_2B_1r_1B_3q_1X_1\bigr)=c\bigl(B_1r_1B_3q_1X_1s_2X_2q_2\bigr)=c\bigl(B_1r_1B_3q_1Xq_2\bigr)\ni Q$.\\
Hence, in this case we have that $(P,Q)\in\mb{LP}_2\bigl(\VV,\de_2(R,S)_{\WW}\bigr)$, and we have associated to these pairs the term $\VV\cdot_{(P,Q)}\de(\WW)_{(R,S)}$.

Note that, in this case
$$\begin{array}{ccl}
\de_{(R,S)}\bigl([\VV,\WW]_{(P,Q)}\bigr) &=&  sign(P,Q)sign(R,S)\bigl\langle\de_1(R,S)\otimes\de_2(R,S)-\de_2(R,S)\otimes\de_1(R,S)\bigr\rangle\\
&=& sign(P,Q)sign(R,S)\bigl\langle c\bigl(r_2B_2s_1B_3q_1X_1\bigr)\otimes c\bigl(s_2X_2p_2Ap_1Xq_2B_1r_1Y\bigr)\\
&& - c\bigl(s_2X_2p_2Ap_1Xq_2B_1r_1Y\bigr)\otimes c\bigl(r_2B_2s_1B_3q_1X_1\bigr)\bigr\rangle.
\end{array}$$
By the other hand
$$\begin{array}{ccl}
                       \VV\cdot_{(P,Q)}\de(\WW)_{(R,S)}  & = & sign(P,Q)sign(R,S)\bigl\langle\de_1(R,S)\otimes\gamma(P,Q)-\gamma(P,Q)\otimes\de_1(R,S)\bigr\rangle \\
                         & = & sign(P,Q)sign(R,S)\bigl\langle c\bigl(r_2B_2s_1B_3q_1X_1\bigr)\otimes c\bigl(p_2Ap_1Xq_2B_1r_1B_3q_1X\bigr)\\
                         && - c\bigl(p_2Ap_1Xq_2B_1r_1B_3q_1X\bigr)\otimes c\bigl(r_2B_2s_1B_3q_1X_1\bigr)\bigr\rangle\\
                         & = & sign(P,Q)sign(R,S)\bigl\langle c\bigl(r_2B_2s_1B_3q_1X_1\bigr)\otimes c\bigl(p_2Ap_1Xq_2B_1r_1Ys_2X_2\bigr)\\
                         && - c\bigl(p_2Ap_1Xq_2B_1r_1Ys_2X_2\bigr)\otimes c\bigl(r_2B_2s_1B_3q_1X_1\bigr)\bigr\rangle.
                      \end{array}$$
Hence $\de_{(R,S)}\bigl([\VV,\WW]_{(P,Q)}\bigr)= \VV\cdot_{(P,Q)}\de(\WW)_{(R,S)}$.\\
$\bullet$ Finally, we suppose that $R\subset\VV$ and $S\subset \WW$. In this case we consider $V_1=p_2A_1r_1Yr_2A_2p_1X$ and $W_1=q_2B_1s_1Y_1q_1Y_2s_2X_1$, where $Y=Y_1q_1Y_2$ and $X=Y_2s_2X_1$. Consequently,
$$\gamma(P,Q)=c\bigl(p_2A_1r_1Yr_2A_2p_1Xq_2B_1s_1Y_1q_1Y_2s_2X_1\bigr)$$
and
$$\begin{array}{ccl}
\de_{(R,S)}\bigl([\VV,\WW]_{(P,Q)}\bigr) &=&  sign(P,Q)sign(R,S)\bigl\langle\de_1(R,S)\otimes\de_2(R,S)-\de_2(R,S)\otimes\de_1(R,S)\bigr\rangle\\
&=& sign(P,Q)sign(R,S)\bigl\langle c\bigl(r_2A_2p_1Xq_2B_1s_1Y\bigr)\otimes c\bigl(s_2X_1p_2A_1r_1Y\bigr)\\
&& - c\bigl(s_2X_1p_2A_1r_1Y\bigr)\otimes c\bigl(r_2A_2p_1Xq_2B_1s_1Y\bigr)\bigr\rangle\\
&=& sign(P,Q)sign(R,S)\bigl\langle c\bigl(r_2A_2p_1Xq_2B_1s_1Y\bigr)\otimes c\bigl(p_2A_1r_1Y_1q_1X\bigr)\\
&& - c\bigl(p_2A_1r_1Y_1q_1X\bigr)\otimes c\bigl(r_2A_2p_1Xq_2B_1s_1Y\bigr)\bigr\rangle\\
\end{array}$$
As $R\subset\VV$ and $S\subset \WW$ we have that $\VV$ and $\WW$ have linear representatives starting at $r_2$ and $s_2$ respectively as follows:
$$V_2=r_2A_2p_1Y_2s_2X_1p_2A_1r_1Y_1q_1Y_2\quad\textrm{and}\quad W_2=s_2X_1q_2B_1s_1Y_1q_1Y_2.$$
 As a consequence $$\gamma(R,S)=c\bigl(r_2A_2p_1Y_2s_2X_1p_2A_1r_1Y_1q_1Y_2s_2X_1q_2B_1s_1Y_1q_1Y_2\bigr)=c\bigl(r_2A_2p_1Xp_2A_1r_1Y_1q_1Xq_2B_1s_1Y\bigr).$$ Therefore, $(P,Q)\in\mb{LP}_1\bigl(\gamma(R,S)\bigr).$ In this situation we can calculate $\de_{(P,Q)}\bigl([\VV,\WW]_{(R,S)}\bigr)$.
 $$\begin{array}{ccl}
     \de_{(P,Q)}\bigl([\VV,\WW]_{(R,S)}\bigr) & = & sign(P,Q)sign(R,S)\bigl\langle\de_1(P,Q)\otimes\de_2(P,Q)-\de_2(P,Q)\otimes\de_1(P,Q)\bigr\rangle \\
      & = &  sign(P,Q)sign(R,S)\bigl\langle c\bigl(p_2A_1r_1Y_1q_1X\bigr)\otimes c\bigl(q_2B_1s_1Yr_2A_2p_1X\bigr)\\
      && -c\bigl(q_2B_1s_1Yr_2A_2p_1X\bigr)\otimes c\bigl(p_2A_1r_1Y_1q_1X\bigr)\bigr\rangle.
         \end{array}
 $$
 Note that $\de_{(R,S)}\bigl([\VV,\WW]_{(P,Q)}\bigr)=-\de_{(P,Q)}\bigl([\VV,\WW]_{(R,S)}\bigr) $, then these terms kill each other in the expression \ref{e1}. As a consequence we have that the compatibility is satisfied.

\end{proof}

\bibliographystyle{amsplain}

\end{document}